\documentclass{amsart}

\usepackage{amsmath, amsthm, amssymb}
\usepackage{verbatim}

\renewcommand{\d}{\partial}
\newcommand{\ddbar}{\sqrt{-1}\d\overline{\d}}

\newtheorem{thm}{Theorem}
\newtheorem{prop}[thm]{Proposition}
\newtheorem{lem}[thm]{Lemma}
\newtheorem{cor}[thm]{Corollary}
\newtheorem{conj}[thm]{Conjecture}

\theoremstyle{definition}
\newtheorem{defn}[thm]{Definition}
\newtheorem{remark}[thm]{Remark}

\renewcommand{\[}{\begin{equation}}
\renewcommand{\]}{\end{equation}}

\title{Fully non-linear elliptic equations on compact Hermitian manifolds}
\author{G\'abor Sz\'ekelyhidi}
\address{Department of Mathematics, University of Notre Dame, Notre
  Dame, IN 46556}
\email{gszekely@nd.edu}

\begin{document} 

\begin{abstract}
We derive a priori estimates for solutions of a general class of fully
non-linear equations on compact Hermitian manifolds.
Our method is based on ideas that have been used
for different specific equations, such as the complex Monge-Amp\`ere, 
 Hessian and inverse Hessian equations. As an application we solve
 a class of Hessian quotient equations on K\"ahler manifolds
 assuming the existence of a
suitable subsolution. The method also applies to analogous equations
on compact Riemannian manifolds. 
\end{abstract}

\maketitle 

\section{Introduction}
Let $(M,\alpha)$ be a compact Hermitian manifold of dimension $n$, and fix a real
$(1,1)$-form $\chi$. For any $C^2$ function $u:M\to\mathbf{R}$ we
obtain a new real $(1,1)$-form $g = \chi + \ddbar u$, and we can
define the endomorphism of $T^{1,0}M$ given by $A^i_j =
\alpha^{i\bar p}g_{j\bar p}$. This is a Hermitian endomorphism with
respect to the metric $\alpha$. We consider equations for $u$ that can
be written in the form
\[ \label{eq:eqn} F(A) = h \]
for a given function $h$ on $M$, where
\[ F(A) = f(\lambda_1,\ldots, \lambda_n) \]
is a smooth symmetric function of the eigenvalues of $A$. Such equations have
been studied extensively in the literature, going back to the work of
Caffarelli-Nirenberg-Spruck~\cite{CNS3} on the Dirichlet problem
in the real case, when $\alpha$
is the Euclidean metric and $M$ is a domain in $\mathbf{R}^n$.

We assume that $f$ is defined in an open symmetric cone
$\Gamma\subsetneq\mathbf{R}^n$, with vertex at the origin, containing the
positive orthant $\Gamma_n$. In addition 
\begin{itemize}
\item[(i)] $f_i > 0$ for all $i$, and $f$ is concave,  
\item[(ii)] $\displaystyle{\sup_{\d \Gamma} f < \inf_M h}$,
\item[(iii)] For any $\sigma < \sup_\Gamma f$ and $\lambda\in \Gamma$
  we have $\lim_{t\to\infty} f(t\lambda) > \sigma$. 
\end{itemize}
Assumption (ii) ensures that the relevant level sets of $f$ do not
intersect the boundary of $\Gamma$. Assumption
(iii) is satisfied by many natural equations, for instance
if $f$ is homogeneous of degree 1 and $f > 0$ in $\Gamma$. 

\begin{defn}\label{defn:subsol}
We say that a smooth function $\underline{u}$ is
a $\mathcal{C}$-subsolution of \eqref{eq:eqn}, if the following
condition holds. At each $x\in M$, define the matrix $B^i_j = \alpha^{i\bar
  p}(\chi_{j\bar p} + \d_j\d_{\bar p} \underline{u})$. Then we require
that for each $x\in M$ the set
\[ \{ \lambda' \in \Gamma\,:\, f(\lambda') = h(x)\,\text{ and }
\lambda' - \lambda(B(x))\in \Gamma_n\} \]
is bounded, where $\lambda(B(x))$ denotes the $n$-tuple of eigenvalues
of $B(x)$. 
\end{defn}

In Section~\ref{sec:subsol} we will describe the relationship between
this notion and that introduced by Guan~\cite{Guan14}. 
Our main result is the following.

\begin{thm}\label{thm:main} Suppose that $u$ is a solution, and
  $\underline{u}$ is a $\mathcal{C}$-subsolution
  of Equation~\ref{eq:eqn}. If we normalize $u$ so that $\sup_M u = 0$,
  then we have an estimate $\Vert
  u\Vert_{C^{2,\alpha}} < C$, where $C$ depends on the given data
  $M,\alpha,\chi,h$, and the subsolution $\underline{u}$. 
\end{thm}

Note that the main result of Guan~\cite{Guan14} is a similar estimate on
Riemannian manifolds, but there the constant $C$ depends in addition
on a $C^1$-bound for $u$. In the Riemannian case this $C^1$-bound can
be obtained under certain extra assumptions, as shown in
\cite{Guan14}, using work of Li~\cite{Li90} and Urbas~\cite{Urb04}. In
addition the subsolution condition in \cite{Guan14} is more
restrictive than ours. As we will discuss in
Section~\ref{sec:Riemannian}, our methods apply with almost no change to
the Riemannian case as well, resulting in an estimate analogous to
Theorem~\ref{thm:main}. 

We first prove
a $C^0$-estimate, generalizing the approach of Blocki~\cite{Bl05_1},
using the Alexandroff-Bakelman-Pucci maximum principle, in the case of
the complex Monge-Amp\`ere equation. For higher order
estimates we use the
method that was employed in the case of the complex Hessian
equations. In other words a $C^1$-bound is derived by combining a second derivative
bound of the form
\[ \label{eq:C2C1} \sup |\d\overline{\d}u| \leq C(1 + \sup |\nabla u|^2), \]
due to Hou-Ma-Wu~\cite{HMW10} in the case of the Hessian equation, with a blowup argument and
Liouville-type theorem due to Dinew-Kolodziej~\cite{DK12}. The gradient
bound combined with \eqref{eq:C2C1} then bounds $|\d\overline{\d}u|$, at which
point the Evans-Krylov theory~\cite{Ev82, Kry82}, adapted to the
complex setting (see for instance
Tosatti-Wang-Weinkove-Yang~\cite{TWWY14})
 can be used to obtain the required
$C^{2,\alpha}$-estimate. Note that as
a consequence of the blowup argument the constant $C$ is not explicit
in Theorem~\ref{thm:main}.

Perhaps the most important equation of the form \eqref{eq:eqn} is
the complex Monge-Amp\`ere equation,
where we take $f = \log\lambda_1\cdot\ldots\cdot\lambda_n$.
The Monge-Amp\`ere
equation was first solved on compact K\"ahler manifolds by
Yau~\cite{Yau78}, and on compact Hermitian manifolds by
Tosatti-Weinkove~\cite{TW10} with some earlier work by
Cherrier~\cite{Ch87}, Hanani~\cite{Ha96} 
and Guan-Li~\cite{GL09}. See also Phong-Song-Sturm~\cite{PSS12_1}
for a recent survey. 
 Note that in this case $\underline{u}$ being a $\mathcal{C}$-subsolution
is equivalent to $\chi + \ddbar \underline{u}$ being
positive definite, and so Theorem~\ref{thm:main} can be used to
recover these existence results.  

A related equation, the Monge-Amp\`ere equation for
$(n-1)$-plurisubharmonic functions, was introduced by
Fu-Wang-Wu~\cite{FWW10}. In
this case we take $f = \log\widetilde{\lambda}_1\cdot\ldots\cdot
\widetilde{\lambda}_n$, where
\[ \widetilde{\lambda}_k = \frac{1}{n-1}\sum_{i\not= k} \lambda_i \]
for each $k$. In terms of matrices, if we define the map $P(M) =
\frac{1}{n-1}(\mathrm{Tr}(M)\cdot I - M)$, then the corresponding
operator $F$ is given by $F(A) = \log\det P(A)$. For this equation,
Theorem~\ref{thm:main} recovers the a priori estimates of
Tosatti-Weinkove~\cite{TW13_1, TW13}, and more general equations in terms of the
$\widetilde{\lambda}_k$ can be considered as well. We will discuss
this in more detail in Section~\ref{sec:ex}. 

A setting when the subsolution property is more subtle is the inverse
$\sigma_k$-equations for $1\leq k \leq n-1$, where we take
\[ f = \left(\frac{\sigma_n}{\sigma_k}\right)^{\frac{1}{n-k}}, \]
for the elementary symmetric functions $\sigma_i$, and the cone
$\Gamma=\Gamma_n$. When $h$ is constant, the equation can be written as
\[ \omega^{n-k}\wedge \alpha^k = c\omega^n, \]
for a constant $c$, where $\omega = \chi + \ddbar u$ is the unknown
metric. When $\alpha, \chi$ are K\"ahler, then we can determine $c$ a priori, since
\[ \label{eq:cfixed} c = \frac{[\omega]^{n-k}\cup [\alpha]^k}{[\omega]^n}. \]
Fixing this value of $c$, if $k=n-1$ then Song-Weinkove~\cite{SW04}
showed that a solution exists if there is a metric
 $\chi' = \chi + \ddbar \underline{u}$
satisfying
\[ \label{eq:SWcone} nc \chi'^{n-1} - (n-1)\chi'^{n-2}\wedge\alpha > 0 \]
in the sense of positivity of $(n-1,n-1)$-forms. This turns out to be
the same as $\underline{u}$ being a $\mathcal{C}$-subsolution. This
result was later generalized by Fang-Lai-Ma~\cite{FLM11} to general $k$,
and existence results for general $k$ and non-constant $h$ on
Hermitian manifolds were obtained by Guan-Sun~\cite{GS13}, Sun~\cite{Sun13}.

Using the continuity method, Theorem~\ref{thm:main} 
can be used to obtain such existence results for
Equation~\eqref{eq:eqn}, under certain assumptions, however it seems
to be difficult to state a satisfactory general existence result, whenever the
subsolution condition is non-trivial in the sense that it depends on
$h$. We give one such result,
Proposition~\ref{prop:riemexist} in the Riemannian case, and the same
proof works in the Hermitian case too. The difficulty is that one needs
extra conditions to ensure that we have a subsolution along the whole
continuity path (see also Guan-Sun~\cite{GS13} for such results in the case
of the inverse $\sigma_k$ equation). The source of this difference between 
Equation~\eqref{eq:eqn} on a compact manifold, and the corresponding Dirichlet
problem,
is that on a compact manifold the constant functions are in the cokernel of
the linearization. 

As an illustration, we will consider general
Hessian quotient equations, of the form
\[ \label{eq:hessquoteq}\omega^l\wedge \alpha^{n-l} = c\omega^k\wedge \alpha^{n-k}, \]
where $(M,\alpha)$ is K\"ahler, $1 \leq l < k\leq n$, the form
$\omega = \chi + \ddbar u$ is the unknown, and $c$ is
determined by
\[ c = \frac{\int_M \chi^l \wedge\alpha^{n-l}}{\int_M \chi^k\wedge
  \alpha^{n-k}}. \]
In analogy with the results of Song-Weinkove and Fang-Lai-Ma for the
case $k=n$, we will show the following.

\begin{cor}\label{cor:hessquot}
  Suppose that there is a form $\chi' = \chi + \ddbar\underline{u}$
  which is $k$-positive (i.e. the eigenvalues satisfy $\sigma_1,
  \ldots, \sigma_k > 0$), and in addition
  \[ kc\chi'^{k-1}\wedge \alpha^{n-k} - l
  \chi'^{l-1}\wedge\alpha^{n-l} > 0 \]
  in the sense of positivity of $(n-1,n-1)$-forms. Then
  \eqref{eq:hessquoteq} has a solution $\omega = \chi + \ddbar u$. 
\end{cor}

There do not seem to be any previous existence results on compact manifolds 
for these equations in the literature when $k < n$, although a priori $C^0$
bounds for the solution $u$ have been found recently by
Sun~\cite{Sun14, Sun14_1}. The
corresponding Dirichlet problem on Euclidean domains does not fit into
the framework of Caffarelli-Nirenberg-Spruck~\cite{CNS3}, but was
subsequently  solved by Trudinger~\cite{Tru95}. 

It is an interesting problem to find geometric assumptions under which
the existence of a $\mathcal{C}$-subsolution can be ensured. In the
case of the Dirichlet problem in Euclidean domains $\Omega$,
Caffarelli-Nirenberg-Spruck~\cite{CNS3} showed that a subsolution 
exists under a suitable convexity type condition on the boundary
$\d\Omega$ (see also Li~\cite{LS04_1} for analogous results in the
complex case).
For the complex Monge-Amp\`ere equation on compact K\"ahler
manifolds,  the result of
Demailly-Paun~\cite{DP04}, characterizing the K\"ahler cone, gives
such a geometric condition. Indeed, this result shows that a real $(1,1)$-class 
$[\chi]$ on a compact K\"ahler manifold $(M,\alpha)$ contains a
K\"ahler metric, if and only if for all analytic subvarieties
$V\subset M$ of dimension $p=1,\ldots, n$ we have
\[ \int_V \chi^k\wedge \alpha^{p-k} > 0, \text{ for }1 \leq k \leq p. \]

In \cite{LSz13}, Lejmi and the author proposed a similar
condition, conjectured to ensure the existence of a metric $\chi' \in
[\chi]$ satisfying the positivity condition \eqref{eq:SWcone}. The
condition is that $[\chi]$ admits a K\"ahler metric, and in addition
\[ \int_V c\chi^p - p \chi^{p-1}\wedge\alpha > 0 \]
for all analytic subvarieties $V\subset M$ of dimension
$p=1,\ldots,n - 1$. For $V=M$ equality has to hold by
\eqref{eq:cfixed}. Recently, in \cite{CSz14}, Collins and the author
resolved this conjecture on toric manifolds. We expect that analogous
results should hold for a large class of equations on K\"ahler
manifolds, and we state a conjecture to this effect for the Hessian
quotient equations in Section~\ref{sec:ex}.
In addition it is natural to expect that for the Dirichlet
problem on K\"ahler manifolds with boundary, the appropriate
subsolutions can be constructed whenever the boundary satisfies a
suitable convexity assumption, and a geometric condition as above is
satisfied for all compact subvarieties of the interior. We hope to
explore such results in future work. 

In Section~\ref{sec:subsol} we give the basic definition and
properties of $\mathcal{C}$-subsolutions. We prove
$C^0$-estimates in Section~\ref{sec:C0}, generalizing the approach of
Blocki~\cite{Bl05_1}. We prove a $C^2$-estimate of the form
\eqref{eq:C2C1} in Section~\ref{sec:C2}, modeled on the work of
Hou-Ma-Wu~\cite{HMW10}. To complete
the proof of Theorem~\ref{thm:main} we use a blowup argument and Liouville-type theorem
analogous to those of Dinew-Kolodziej~\cite{DK12} in
Sections~\ref{sec:DK}, \ref{sec:blowup}.  In
Section~\ref{sec:ex} we give the proof of
Corollary~\ref{cor:hessquot}. Finally in Section~\ref{sec:Riemannian}
we discuss analogous problems on compact Riemannian manifolds. 

\section{Subsolutions}\label{sec:subsol}
As in the introduction, let $\Gamma\subsetneq \mathbf{R}^n$ be a
symmetric, open, convex cone with vertex at the origin, containing the positive
orthant $\Gamma_n$, and let $f : \Gamma\to\mathbf{R}$ be a smooth,
concave function, satisfying the monotonicity condition $f_i > 0$ for
all $i$. We denote by $\mathcal{F}$ the function $\mathcal{F}(\lambda)
= \sum_i f_i(\lambda)$. 

Define
\[ \sup_{\d \Gamma} = \sup_{\lambda' \in \d\Gamma}
\limsup_{\lambda\to\lambda'} f(\lambda). \]
For any $\sigma > \sup_{\d\Gamma}$, the set
\[ \Gamma^\sigma = \{\lambda\,:\, f(\lambda) > \sigma\} \]
is a convex open set. Fix a value of $\sigma$ for which
$\Gamma^\sigma \ne \emptyset$. Then the level set $\d \Gamma^\sigma =
f^{-1}(\sigma)$ is a smooth hypersurface. In view of
Definition~\ref{defn:subsol} we are interested in those
$\mu\in\mathbf{R}^n$, for which the set $(\mu + \Gamma_n)\cap
\d\Gamma^\sigma$ is bounded. These $\mu$ represent the possible
eigenvalues of a $\mathcal{C}$-subsolution.

For any $\lambda\in \d\Gamma^\sigma$ let us
write $\mathbf{n}_\lambda$ for the inward pointing unit normal vector, i.e. 
\[ \mathbf{n}_\lambda = \frac{\nabla f}{|\nabla f|} \]
Note that since $f_i > 0$ for all $i$, we have
\[ \sum_{i=1}^n f_i(\lambda)^2 \leq \left(\sum_{i=1}^n
  f_i(\lambda)\right)^2 \leq n \sum_{i=1}^n f_i(\lambda)^2, \]
and so
\[ \label{eq:s2} |\nabla f| \leq \mathcal{F} \leq \sqrt{n}|\nabla f|. \]
In particular the unit normal $\mathbf{n}$ is a bounded multiple of
$\mathcal{F}^{-1} \nabla f$. 

\begin{remark}
  Using this setup,  Guan~\cite{Guan14} introduced a convex open set
  $\mathcal{C}^+_\sigma\subset \Gamma$, which consists of those $\mu$
  for which the set
  \[ \d\Gamma^\sigma(\mu) = \{ \lambda\in \d\Gamma^\sigma\,:\, (\lambda - \mu)\cdot
  \mathbf{n}_\lambda > 0 \} \]
  is bounded. In turn this leads to a notion of subsolution for the
  equation $F(A)=h$ similar to Definition~\ref{defn:subsol}, except one
  requires that $\lambda(B)\in \mathcal{C}^+_{h(x)}$ at each point.
  Since $\mathbf{n}$ has positive entries, we have
  \[ (\mu + \Gamma_n)\cap \d\Gamma^\sigma \subset
  \d\Gamma^\sigma(\mu), \]
  and so this notion of subsolution is more restrictive than that of a
  $\mathcal{C}$-subsolution. As a simple example, if $\d
  \Gamma^\sigma$ is a hyperplane, as is the case for a linear elliptic
  equation, then $\mathcal{C}^+_\sigma$ is
  empty, but $(\mu + \Gamma_n)\cap \d\Gamma^\sigma$ is bounded for all
  $\mu\in \mathbf{R}^n$. 
\end{remark}

The main result that we need
is the following, which is a refinement of \cite[Theorem
2.16]{Guan14}. 

\begin{prop}\label{prop:subsol1}
  Suppose that $\mu\in \mathbf{R}^n$ is such that for some $\delta, R > 0$
  \[ \label{eq:s1}
  (\mu - 2\delta\mathbf{1} + \Gamma_n)\cap \d\Gamma^\sigma \subset B_R(0),\]
  where $B_R(0)$ is the ball of radius $R$ around the origin.

  Then there is a constant $\kappa > 0$ depending on $\delta$ and on
  the set in \eqref{eq:s1} (more precisely the normal vectors of
  $\d\Gamma^\sigma$ on this set),
  such that if $\lambda\in \d\Gamma^\sigma$ and $|\lambda| >
  R$, then either
  \[ \label{eq:s3} \sum_{i=1}^n f_i(\lambda)(\mu_i - \lambda_i) > \kappa\mathcal{F}(\lambda), \]
  or $f_i(\lambda) > \kappa \mathcal{F}(\lambda)$ for all $i$.
\end{prop}
\begin{proof}
  Consider the set
  \[ A_\delta = \{ v\in \Gamma\,:\, f(v) \leq \sigma,\text{ and } v -
  \mu - \delta\mathbf{1}\in \overline{\Gamma_n}\}. \]
  Because of $\eqref{eq:s1}$ this is a compact set. For each $v \in A_\delta$
  consider the cone $\mathcal{C}_v$ with vertex at the origin defined
  by
  \[ \mathcal{C}_v = \{ w\in \mathbf{R}^n\,:\, v + tw\in (\mu -
  2\delta\mathbf{1} + \Gamma_n)\cap \d\Gamma^\sigma\,\text{ for some }
  t > 0\}. \]
  In other words the cone $v + \mathcal{C}_v$ has vertex $v$ and cross
  section $(\mu - 2\delta\mathbf{1} + \Gamma_n)\cap
  \d\Gamma^\sigma$. Since $f_i > 0$ for all $i$, the set $(\mu -
  2\delta\mathbf{1} + \Gamma_n)\cap \d\Gamma^\sigma$ is strictly larger than $(\mu -
  \delta\mathbf{1} + \Gamma_n)\cap \d\Gamma^\sigma$, i.e.
  \[ \overline{(\mu -
  \delta\mathbf{1} + \Gamma_n)\cap \d\Gamma^\sigma} \subset (\mu -
2\delta\mathbf{1} + \Gamma_n)\cap \d\Gamma^\sigma. \]
  This implies that the cone $\mathcal{C}_v$ is strictly larger than
  $\Gamma_n$.   Let us
  denote by $\mathcal{C}_v^*$ the dual cone, i.e.
  \[ \mathcal{C}_v^* = \{x\in \mathbf{R}^n\,:\, \langle x,y\rangle >
  0\,\text{ for all }y\in \mathcal{C}_v\}. \]
  Being strictly larger than $\Gamma_n$ means that there is an 
  $\epsilon > 0$ such that if $x\in \mathcal{C}_v^*$ is a unit vector,
  then each entry of $x$ satisfies $x_i > \epsilon$. Since $A_\delta$
  is compact, we can choose a uniform $\epsilon$ that works for all
  $v\in A_\delta$. 

  Suppose that $\lambda\in \d\Gamma^\sigma$, and $|\lambda| > R$. Let
  $T_\lambda$ be the tangent plane to $\d\Gamma^\sigma$ at
  $\lambda$. There are two possibilities:
  \begin{itemize}
  \item If $T_\lambda$ intersects $A_\delta$, in a point $v$, say,
    then the cone $v + \mathcal{C}_v$ lies above
    $T_\lambda$ (i.e. $\Gamma^\sigma$ lies on the same side of
    $T_\lambda$  as $v + \mathcal{C}_v$). This
    implies that the normal vector of $T_\lambda$ is in the dual cone,
    i.e. $\mathbf{n}_\lambda \in \mathcal{C}_v^*$. But then each entry
    of $\mathbf{n}_\lambda$ is greater than $\epsilon$, i.e. $f_i >
    \epsilon |\nabla f|$ for each $i$. Because of \eqref{eq:s2} this implies
    \[ f_i > \frac{\epsilon}{\sqrt{n}}\mathcal{F} \]
    for each $i$. 
  \item If $T_\lambda$ does not intersect $A_\delta$, then $\mu$
    must be of distance at least $\delta$ from $T_\lambda$. This means
    that $(\mu - \lambda)\cdot \mathbf{n}_\lambda > \delta$. Writing
    this out in components, we have
    \[ \sum_{i=1}^n f_i(\lambda)\,(\mu_i - \lambda_i) > \delta |\nabla
    f(\lambda)|, \]
    which by \eqref{eq:s2} implies \eqref{eq:s3}. 
  \end{itemize}
\end{proof}

We need to apply this to the function $F$ defined on the space of Hermitian
matrices $A$ by $F(A) = f(\lambda(A))$, where
\[ \lambda(A) = (\lambda_1,\ldots, \lambda_n) \]
denotes the eigenvalues of $A$. Let us write $F^{ij}$ for the
derivative of $F$ with respect to the $ij$-entry of $A$. Then
similarly to Guan~\cite[Theorem 2.18]{Guan14},  we have
the following. 

\begin{prop}\label{prop:subsol2}
  Let $[a,b]\subset (\sup_{\d\Gamma}f, \sup_\Gamma f)$ and $\delta, R >
  0$. There exists $\kappa > 0$ with the following property. Suppose
  that $\sigma\in [a,b]$ and $B$ is a Hermitian matrix such that
   \[ \label{eq:subsol10}
 (\lambda(B) - 2\delta\mathbf{1} + \Gamma_n)\cap \d\Gamma^\sigma
   \subset B_R(0). \]
  Then for any Hermitian matrix
  $A$ with $\lambda(A)\in \d\Gamma^\sigma$ and $|\lambda(A)| > R$ 
  we either have 
  \[ \label{eq:subsol}
     \sum_{p,q} F^{pq}(A)\,\big[ B_{pq} - A_{pq}\big] > \kappa\sum_p F^{pp}(A), \]
  or $F^{ii}(A) > \kappa\sum_p F^{pp}(A)$ for all $i$. 
\end{prop}
\begin{proof}
The proof is essentially the same as that of \cite[Theorem
2.18]{Guan14}, but we give some details for the reader's
convenience. Suppose that $A$ is diagonal, and its eigenvalues satisfy
$\lambda_1 \geq \lambda_2 \geq \ldots \geq \lambda_n$. This implies
that $F^{pq} = 0$ if $p\not=q$, and that $F^{11} \leq F^{22}
\leq\ldots \leq F^{nn}$. Let $\mu_1,\ldots, \mu_n$ be the eigenvalues
of $B$ ordered so that $\mu_1 \geq \mu_2 \geq\ldots \geq \mu_n$. 
The matrix $B$ may not be
diagonal, but the Schur-Horn theorem implies that the $n$-tuple of
diagonal entries $(B_{11},\ldots, B_{nn})$ is in the convex hull of
the vectors obtained by permuting the entries of $(\mu_1,\ldots,
\mu_n)$. In particular it follows that
\[ \sum_i F^{ii}(A) B_{ii} \geq F^{ii}(A) \mu_i. \]
Since $A$ is diagonal, we have $F^{ii} = f_i(\lambda)$, and
\eqref{eq:subsol10} implies that we can apply
Proposition~\ref{prop:subsol1} to obtain the required inequalities. 
We obtain uniform $\kappa > 0$, since the assumptions on $\sigma$
implies that the sets $(\lambda(B) - 2\delta\mathbf{1} + \Gamma_n)\cap
\d\Gamma^\sigma$ move in a compact family.
\end{proof}

We recall the definition of a $\mathcal{C}$-subsolution from the introduction. 
\begin{defn}
   Suppose, as in the introduction that $(M,\alpha)$ is Hermitian and
   $\chi$ is a real $(1,1)$-form. We say that $\underline{u}$ is a
   $\mathcal{C}$-subsolution for the equation $F(A) = h$, if at each
   $x\in M$ the set
   \[ \left(\lambda\big[\alpha^{j\bar p}(\chi_{i\bar p} + \underline{u}_{i\bar p})\big] +
   \Gamma_n\right)\cap \d\Gamma^{h(x)} \]
   is bounded. Let us also say that $\underline{u}$ is admissible, if
   $\lambda\big[\alpha^{j\bar p}(\chi_{i\bar p} + \underline{u}_{i\bar
     p})\big] \in \Gamma$. An observation pointed out to the author by
   Wei Sun is that a $\mathcal{C}$-subsolution need not be
   admissible. 
\end{defn}

\begin{remark}\label{rem:subsol2}
In examples it is useful to have an alternative description of the set
of $\mathcal{C}$-subsolutions. Following
Trudinger~\cite{Tru95}, let us denote by $\Gamma_\infty$ the
projection of $\Gamma$ onto $\mathbf{R}^{n-1}$:
\[ \Gamma_\infty = \{(\lambda_1,\ldots, \lambda_{n-1})\,:\,
(\lambda_1,\ldots, \lambda_n)\in \Gamma\,\text{ for some }
\lambda_n\}.\]
For $\mu\in \mathbf{R}^n$, the set $(\mu + \Gamma_n)\cap \d\Gamma^\sigma$ is
bounded, if and only if
\[  \lim_{t\to\infty} f(\mu + t\mathbf{e}_i) > \sigma \]
for all $i$, where $\mathbf{e}_i$ is the $i^{\text{th}}$ standard
basis vector. Note that this limit is defined as long as any
$(n-1)$-tuple $\mu'$ in $\mu$ satisfies $\mu'\in \Gamma_\infty$. In
other words it is defined for $\mu\in \widetilde{\Gamma}$, where
$\widetilde{\Gamma}\subset\mathbf{R}^n$ is given by
\[ \widetilde{\Gamma} = \{ \mu\in\mathbf{R}^n\,:\, \text{ there exists
}t > 0\text{ such that } \mu + t\mathbf{e}_i\in\Gamma\text{ for all
}i\}. \]
Note that $\Gamma \subset \widetilde{\Gamma}$. 

For any $\lambda' = (\lambda_1,\ldots, \lambda_{n-1})\in \Gamma_\infty$, consider the limit
\[ \lim_{\lambda_n\to\infty} f(\lambda_1,\ldots, \lambda_n). \]
Then as in \cite{Tru95} this limit is either finite for all $\lambda'$
or infinite for all $\lambda'$ because of the concavity of $f$. If the
limit is infinite, then  $(\mu + \Gamma_n)\cap \d\Gamma^\sigma$ is
bounded for all $\sigma$ and $\mu\in \widetilde{\Gamma}$. In
particular any admissible $\underline{u}$ is a
$\mathcal{C}$-subsolution in this case. 

If the limit is finite, define the
function $f_\infty$ on $\Gamma_\infty$ by  
\[ f_\infty(\lambda_1,\ldots,\lambda_{n-1}) =
\lim_{\lambda_n\to\infty} f(\lambda_1,\ldots,\lambda_n).  \]
From the above it is clear that $(\mu + \Gamma_n)\cap \d\Gamma^\sigma$
for $\mu\in\widetilde{\Gamma}$ 
is bounded if and only if $f_\infty(\mu') > \sigma$, where $\mu'\in
\Gamma_\infty$ denotes any $(n-1)$-tuple of entries of $\mu$.
\end{remark}

We will need the following consequences of our
structural assumptions for $f$.

\begin{lem}\label{lem:struct}
  Under the assumptions (i), (ii), (iii) for $f$ in the introduction,
  we have the following, for any $\sigma\in (\sup_{\d\Gamma} f,
  \sup_\Gamma f)$: 
  \begin{enumerate}
  \item[(a)] There is an $N > 0$ depending on $\sigma$, such that
    $\Gamma + N\mathbf{1} \subset \Gamma^\sigma$,
 \item[(b)] there is a $\tau  >0$, depending on $\sigma$, such that
    $\mathcal{F}(\lambda) > \tau$ for any $\lambda\in\d\Gamma^\sigma$.
  \end{enumerate}
\end{lem}
\begin{proof}
  To prove (a), let $x\in \d\Gamma^\sigma$ be the closest point to the
  origin. By the convexity of $\Gamma^\sigma$ and symmetry under
  permuting the variables, we must have $x = N\mathbf{1}$ for some $N
  > 0$. We claim that $\Gamma + N\mathbf{1} \subset
  \Gamma^\sigma$. Indeed for any $\lambda\in \Gamma$,
  assumption (iii) implies that there is some $T > 1$, such that
  $T\lambda\in \Gamma^\sigma$. The convexity of
  $\Gamma^\sigma$ implies that then $x + t\lambda\in \Gamma^\sigma$
  for all $t \in (0,T]$, and so in particular $x + \lambda\in
  \Gamma^\sigma$. This proves (a).

   To prove (b), first choose $\sigma' > \sigma$ such that
   $\sigma'\in(\sup_{\d\Gamma} f,\sup_\Gamma f)$ as well.
   Part (a) implies that if $f(\lambda) = \sigma$, then $f(\lambda + N\mathbf{1})
   > \sigma'$. By concavity we have
   \[ f(\lambda + N\mathbf{1}) \leq f(\lambda) + N\sum_{i=1}^n
   f_i(\lambda), \]
   which implies $\mathcal{F}(\lambda) \geq N^{-1}(\sigma' -
   \sigma)$, which is the bound that we wanted. 
\end{proof}

\section{$C^0$-estimates}\label{sec:C0}
In this section we prove a priori $C^0$-estimates for solutions of
Equation~\eqref{eq:eqn}.

\begin{prop}\label{prop:C0}
  Suppose that $F(A) = h$, where $A^{ij} = \alpha^{j\bar
  p}g_{i\bar p}$ and $g = \chi + \ddbar u$ for a fixed background form
$\chi$, as in the introduction. Assume that we have a $\mathcal{C}$-subsolution
$\underline{u}$, and normalize $u$ so that $\sup u -
\underline{u} = 0$. There is a constant $C$, depending on the given
data, including $\underline{u}$ such that
\[ \sup_M |u| < C. \]
\end{prop}
\begin{proof}
  Our proof is based on the method that Blocki~\cite{Bl05_1} used in the
  case of the complex Monge-Amp\`ere equation. To simplify notation we can
  assume $\underline{u}=0$, by changing $\chi$. We therefore have
  $\sup_M u =0$, and our goal is to obtain a lower bound for $L = \inf_M u$

Note that our assumptions for $\Gamma$ imply (see \cite{CNS3}) that
\[ \Gamma \subset \{(\lambda_1,\ldots,\lambda_n)\,:\, \sum_i \lambda_i
> 0\}, \]
which in turn implies that $\mathrm{tr}_\alpha g > 0$. It follows that
we have a lower bound for $\Delta u$ (the complex Laplacian with
respect to $\alpha$), and so using the Green's
function of a
Gauduchon metric conformal to $\alpha$ as in
Tosatti-Weinkove~\cite{TW10_1},
we have a uniform bound for $\Vert u \Vert_{L^1}$.

An alternative,
more local argument for obtaining a bound for $\Vert |u|^p\Vert_{L^1}$
for some $p >0$ follows from the weak
Harnack inequality, Gilbarg-Trudinger~\cite[Theorem
9.22]{GT98}. Indeed, suppose that we cover $M$ with a finite number of coordinate balls
$2B_i$, so that the balls $B_i$ of half the radius still cover $M$.
Using the lower bound $\Delta u \geq -C$, and the fact
that $u$ is non-positive, we can apply the weak Harnack inequality in
each ball $2B_i$ to obtain an estimate of the form
\[ \label{eq:wH}
\left( \int_{B_i} (-u)^p \right)^{1/p} \leq C_1\big[\inf_{B_i} (-u) +
1\big], \]
with $p, C_1 > 0$ depending only on the covering and the metric
$\alpha$. Our assumption $\sup_M u = 0$ implies that we have one ball,
say $B_0$ in which $\inf_{B_0} (-u) = 0$. Then \eqref{eq:wH} implies a bound on the
integral of $|u|^p$ on $B_0$. But this in turn gives a bound
$\inf_{B_i} |u| < C_2$, for all balls $B_i$ with $B_i\cap B_0 \ne
\emptyset$, and then \eqref{eq:wH} can be used again to bound the
integrals of $|u|^p$ on each such $B_i$. Continuing in this way, we
will end up with bounds on each ball, and so we obtain an a priori
bound for $\Vert |u|^p\Vert_{L^1}$ on 
$M$. 

Being a $\mathcal{C}$-subsolution means that for each $x$ the set
\[ \left(\lambda(\alpha^{j\bar p}\chi_{i\bar p}) + \Gamma_n\right)\cap
\d\Gamma^{h(x)}\]
is bounded. There is then a $\delta > 0$ and $R > 0$
such that at each $x$ we have
\[ \label{eq:c1} (\lambda(\alpha^{j\bar p}\chi_{i\bar p}) - \delta
\mathbf{1} + \Gamma_n) \cap \d\Gamma^{h(x)} \subset B_R(0). \]

Let us work in local coordinates $z_i$, for which the infimum $L$ is
achieved at the origin, and the coordinates are defined for $|z_i| <
1$, say. We write $B(1) = \{z\,:\, |z| < 1\}$. Let $v = u + \epsilon
|z|^2$ for a small $\epsilon > 0$. We have $\inf v = L =
v(0)$, and $v(z) \geq L + \epsilon$ for $z\in \d B(1)$.  From
Proposition~\ref{prop:abp} we obtain
\[\label{eq:c3} c_0 \epsilon^{2n} \leq \int_{P} \det(D^2v), \]
where $P$ is defined as in \eqref{eq:contactset}. 
As in Blocki~\cite{Bl05_1}, at any point $x\in P$ we have $D^2v(x) \geq
0$ and so
\[ \label{eq:c4} \det(D^2v) \leq 2^{2n}\det(v_{i\bar j})^2. \]
At the same time, if $x\in P$, then $D^2v(x) \geq 0$ implies that
$u_{i\bar j}(x) \geq -\epsilon \delta_{i\bar j}$. If $\epsilon$ is
sufficiently small (depending on the metric $\alpha$ and the choice of
$\delta$), then this implies that at $x\in P$
\[ \label{eq:c8} \lambda\big[\alpha^{j\bar p}(\chi_{i\bar p} +
u_{i\bar p})\big] \in \lambda(\alpha^{j\bar p}\chi_{i\bar p}) - \delta\mathbf{1} + \Gamma_n. \]
According to the equation $F(A) = h$ we also have
$\lambda\big[\alpha^{j\bar p}(\chi_{i\bar p} +
u_{i\bar p})\big]\in \d\Gamma^{h(x)}$ at $x$, so from \eqref{eq:c1} we get
an upper bound $|u_{i\bar j}| < C$. This
gives a bound for $v_{i\bar j}$ at any $x\in P$, so from \eqref{eq:c4} and
\eqref{eq:c3} we get
\[ \label{eq:c5} c_0\epsilon^{2n} \leq C' \mathrm{vol}(P). \]
By definition, for $x\in P$ we have $v(0) > v(x) - \epsilon / 2$,
and so $v(x) < L + \epsilon / 2$. This implies 
\[ \mathrm{vol}(P) \leq \frac{ \Vert |v|^p\Vert_{L^1} }{\left| L +
    \frac{\epsilon}{2}\right|^p}. \]
Since we already have a bound for $\Vert |v|^p\Vert_{L^1}$, this
inequality contradicts \eqref{eq:c5} if $|L|$ is very large. 
\end{proof}

We used the following variant of the Alexandroff-Bakelman-Pucci
maximum principle, similar to Gilbarg-Trudinger~\cite[Lemma 9.2]{GT98}. 
\begin{prop}\label{prop:abp}
Let $v : B(1) \to \mathbf{R}$ be smooth, such that $v(0) + \epsilon \leq 
 \inf_{\d B(1)} v$, where $B(1)$ denotes the unit ball in
$\mathbf{R}^n$. Define the set
\[ \label{eq:contactset} P = \left\{ x\in B(1)\,:\, \begin{minipage}{3in} $$\begin{aligned}
&|Dv(x)| < \frac{\epsilon}{2},\text{ and } \\
&v(y) \geq v(x) + Dv(x)\cdot (y-x)\,\text{ for all }y\in B(1)
\end{aligned} $$\end{minipage} \right\}. \]
Then for a dimensional constant $c_0 > 0$ we have
\[ \label{eq:abp} c_0 \epsilon^{n} \leq \int_P \det(D^2 v). \]
\end{prop}
\begin{proof}
 The proof follows the argument of \cite[Lemma 9.2]{GT98}. Consider
 the graph of $v$, and let $\xi\in \mathbf{R}^n$ be such that $|\xi| <
 \frac{\epsilon}{2}$. The graph of the function $l(x) = v(0) +
 \xi\cdot x$ lies below the graph of $v$ on the boundary $\d B(1)$ by
 our assumption on $v$, and it intersects the graph of $v$ at
 $(0,v(0))$. This implies that for some $k \geq 0$, the graph of $l(x) -
 k$ is tangent to $v$ at some point $x\in B(1)$, and considering the
 largest such $k$ we will have $x\in P$. In particular the ball
 $B(\epsilon / 2)$ is in the image of $P$ under the gradient of $v$,
 i.e. $B(\epsilon / 2) \subset \nabla v(P)$. The inequality
 \eqref{eq:abp} follows by comparing volumes. 
\end{proof}

\begin{remark}
This method can also be used to obtain $C^0$-estimates for more
general types of equations, where the matrix $A$ in the
equation $F(A) = h$ depends on the gradient of $u$ as well. We
illustrate this with an example taken from the recent work of
Tosatti-Weinkove~\cite{TW13} on $(n-1)$-plurisubharmonic
functions. On a Hermitian manifold $(M,\alpha)$, given another
Hermitian metric $\chi$ the equation can be
written as
\[\label{eq:TWeqn}
\det\left( \chi + \frac{1}{n-1}\big[ (\Delta u)\alpha -
  \ddbar u\big] + \ast E\right) = e^h \det\alpha, \]
where $\ast$ is the Hodge star operator of $\alpha$ and
\[ E = \frac{1}{(n-1)!} \mathrm{Re}\big[\sqrt{-1}\d u\wedge
\overline{\d}(\alpha^{n-2})\big]. \]
In addition $\Delta$ is the complex Laplacian with respect to $\alpha$. 
It is assumed that $\alpha$ is a Gauduchon metric, and the form inside
the determinant in \eqref{eq:TWeqn} is positive definite. Normalizing
$u$ so that $\sup_M u = 0$, it is shown
in \cite{TW13} that this implies an $L^1$-bound $\Vert u\Vert_{L^1} <
C$. An argument using the weak Harnack inequality as above can also be
used to bound $\Vert |u|^p\Vert_{L^1}$ for some $p> 0$.
In \cite{TW13} this is then used together with a Moser iteration argument to
bound $\inf_M u$. We can obtain a different proof of this bound along
the lines of the argument above. 

As in the proof of Proposition~\ref{prop:C0},
choose coordinates $z_i$ in which the infimum $L = \inf_M u$ is
achieved at the origin and for a small $\epsilon > 0$ we let $v = u +
\epsilon|z|^2$. We apply Proposition~\ref{prop:abp} to obtain
\[ \label{eq:c9} c_0 \epsilon^{2n} \leq \int_P \det(D^2 v), \]
with the set $P$ in \eqref{eq:contactset}. By definition, if $x\in P$,
then we have $D^2v \geq 0$ and so $u_{i\bar j}(x) \geq -\epsilon
\delta_{ij}$, and in addition $|Dv(x)| < \epsilon / 2$. If $\epsilon$ is chosen sufficiently
small (depending on $\chi$ and $\alpha$), then 
Equation~\eqref{eq:TWeqn} implies an upper bound for $u_{i\bar j}(x)$
at all $x\in P$, using that $\chi$ is positive definite. From \eqref{eq:c9} we
then get
\[ c_0 \epsilon^{2n} \leq C'\mathrm{vol}(P), \]
but as before,
\[ \mathrm{vol}(P) \leq \frac{ \Vert |v|^p\Vert_{L^1}}{ \left| L +
    \frac{\epsilon}{2}\right|^p}, \]
which is a contradiction if $|L|$ is too large.

It is an interesting problem whether the $C^2$-estimate in
Section~\ref{sec:C2} can also be extended to more general equations
$F(A) = h$, where $A$ depends on the gradient of $u$ as well. In
particular this estimate is not known at present for
Equation~\eqref{eq:TWeqn}. Note that in the Riemannian case
$C^2$-estimates for such equations, under certain conditions, are derived in
Guan-Jiao~\cite{GJ14}. 
\end{remark}

\section{$C^2$-estimates}\label{sec:C2}
Our goal in this section is the following estimate for the complex Hessian of
$u$ in terms of the gradient. As in the introduction we assume that
$u$ satisfies an equation of the form $F(A) = h$, where $A^{ij} =
\alpha^{j\bar p}g_{i\bar p}$ and $g_{i\bar j} =\chi_{i\bar j} +
u_{i\bar j}$ for a given form $\chi$. In addition we assume the
existence of a $\mathcal{C}$-subsolution $\underline{u}$. 

\begin{prop}\label{prop:HMW}
  We have an estimate
  \[ |\d \overline{\d} u| \leq C(1 + \sup_M |\nabla u|^2_\alpha), \]
  where the constant depends on the background
  data, in particular $\Vert\alpha\Vert_{C^2}$, $\Vert h\Vert_{C^2}$,
  $\Vert\chi\Vert_{C^2}$ 
  and the subsolution $\underline{u}$. 
\end{prop}

To simplify notation, we will assume that the subsolution
$\underline{u}=0$, since otherwise we could simply modify the
background form $\chi$. By definition this means that for each $x\in
M$ the sets $(\lambda(B(x)) + \Gamma_n)\cap \d\Gamma^{h(x)}$ are
bounded, where $B^{ij} = \alpha^{j\bar p}\chi_{i\bar p}$. We can find
$\delta, R > 0$ such that at each $x$, 
\[ \label{eq:subsol3} (\lambda(B) - 2\delta\mathbf{1} + \Gamma_n)\cap \d\Gamma^{h(x)}
\subset B_R(0). \]
In particular, by Proposition~\ref{prop:subsol2} we have a $\kappa >
0$ with the following property: at any
$x\in M$, if $|\lambda(A)| > R$ and $A$ is diagonal with eigenvalues
$\lambda_1,\ldots, \lambda_n$, then either 
\[ \label{eq:alt1} F^{ii}(A) > \kappa\sum_p
F^{pp}(A) \text{ for all $i$}, \]
or 
\[ \label{eq:subsol2}
\sum_p F^{pp}(A)\big[ B^{pp} - \lambda_p\big] > \kappa\sum_p
F^{pp}(A). \]
Also, Lemma~\ref{lem:struct} implies that we have a constant
$\tau > 0$ such that $\sum_p F^{pp}(A) > \tau$. 

Our calculation will mostly follow that of Hou-Ma-Wu~\cite{HMW10} which in
turn is based on ideas in Chou-Wang~\cite{CW01}. One key difference is
that instead of using $g_{1\bar 1}$ in suitable coordinates, we use
the maximum eigenvalue of the matrix $A$. This introduces good
positive terms which are useful in the Hermitian case, somewhat
similar to the extra term $\log \alpha^{p\bar q}g_{1\bar q}g_{p\bar
  1}$ introduced in Tosatti-Weinkove~\cite{TW13}. 
The idea of
exploiting the inequality \eqref{eq:subsol2} is from
Guan~\cite{Guan14}. A refinement of this also appears in
Guan~\cite{Guan14_1} where the two possibilities \eqref{eq:alt1} and
\eqref{eq:subsol2} are exploited, although the setup is not the same
as ours. 

We first review some basic formulas for the derivatives of eigenvalues
which can be found in Spruck~\cite{Spruck05} for instance. 
The derivatives of the eigenvalue $\lambda_i$ at a diagonal
matrix with distinct eigenvalues are
\[ \lambda_i^{pq} = \delta_{pi}\delta_{qi} \]
\[ \lambda_i^{pq,rs} =
(1-\delta_{ip})\frac{\delta_{iq}\delta_{ir}\delta_{ps}}{\lambda_i -
  \lambda_p} +
(1-\delta_{ir})\frac{\delta_{is}\delta_{ip}\delta_{rq}}{\lambda_i -
  \lambda_r}, \]
where $\lambda_i^{pq}$ denotes the derivative with respect to the
$pq$-entry. 

It follows from this that for any symbols $A^{pq}_k$ we have 
\[ \label{eq:l1}
\lambda_1^{pq,rs}A^{pq}_k A^{rs}_{\bar k} = \sum_{p > 1}
\frac{A^{p1}_kA^{1p}_{\bar k} + A^{1p}_kA^{p1}_{\bar k}}{\lambda_1 -
  \lambda_p}.\]

If $F(A) = f(\lambda_1,\ldots, \lambda_n)$ in terms of a symmetric
function of the eigenvalues, then at a diagonal matrix $A$ with
distinct eigenvalues we have (see also Gerhardt~\cite{Ge96})
\[ F^{ij} = \delta_{ij}f_i \]
\[ F^{ij, rs} = f_{ir}\delta_{ij}\delta_{rs} + \frac{f_i-f_j}{\lambda_i
  -\lambda_j}(1 - \delta_{ij})\delta_{is}\delta_{jr}. \]
Note that these formulas make sense even when the eigenvalues are not
distinct, since $F$ is a smooth function on the space of matrices if
$f$ is symmetric. In particular as $\lambda_i\to \lambda_j$ we also
have $f_i \to f_j$. 
It follows that 
\[\begin{aligned} \label{eq:concaveF}
  F^{ij,rs}u_{i\bar jk}u_{r\bar s\bar k} &= f_{ij}u_{i\bar
    ik}u_{j\bar j\bar k} + \sum_{p\ne q}
  \frac{f_p-f_q}{\lambda_p - \lambda_q}|u_{p\bar qk}|^2 \\
  &\leq f_{ij}u_{i\bar ik}u_{j\bar j\bar k} + \sum_{i > 1}
  \frac{f_1-f_i}{\lambda_1 - \lambda_i} |u_{i\bar 1k}|^2,
\end{aligned}\]
since if  $f$ is concave and symmetric, one can show (see
Spruck~\cite{Spruck05}) that $\frac{f_i - f_j}{\lambda_i -
  \lambda_j}\leq 0$. In particular $f_i \leq f_j$ if $\lambda_i \geq
\lambda_j$. 

We want to apply the maximum principle to a function $G$ of the form
\[ G = \log \lambda_1 + \phi(|\nabla u|^2) + \psi(u), \]
where $\lambda_1: M\to\mathbf{R}$ is the largest eigenvalue of the
matrix $A$ at each point. Since the eigenvalues of $A$ need not be
distinct at the point where $G$ achieves its maximum, we will perturb
$A$ slightly. 

To do this, choose local coordinates $z_i$, such that $G$ achieves its
maximum at the origin, and at the origin $A$ is diagonal with
eigenvalues $\lambda_1 \geq \lambda_2 \geq\ldots \geq \lambda_n$. Let
$B$ be a diagonal matrix such that $B^{11} = 0$ and $0 < B^{22} <
\ldots < B^{nn}$ are small, satisfying $B^{nn} < 2B^{22}$. 
Define the matrix $\widetilde{A} = A -
B$. At the origin, $\widetilde{A}$ has eigenvalues
\[ \widetilde{\lambda_1} = \lambda_1, \quad \widetilde{\lambda_i} =
\lambda_i - B^{ii}\,\text{ if }i > 1. \]
Since these are distinct, the eigenvalues of $\widetilde{A}$ define
smooth functions near the origin. 

In the calculations below we use derivatives with respect to the Chern
connection of $\alpha$. 
From the formulas for the derivatives of the $\widetilde{\lambda}_i$, we have
\[ \label{eq:ff1} \begin{aligned}
\widetilde{\lambda}_{1, k} &= \widetilde{\lambda}_1^{pq}(\widetilde{A}^{pq})_k = g_{1\bar 1k}
- (B^{11})_k \\
\widetilde{\lambda}_{1, k\bar k} &= \widetilde{\lambda}_1^{pq,rs}(\widetilde{A}^{pq})_k
(\widetilde{A}^{rs})_{\bar k} +
\widetilde{\lambda}_1^{pq}(\widetilde{A}^{pq})_{k\bar k} \\
&= g_{1\bar 1k\bar k} + \sum_{p > 1} \frac{|g_{p\bar 1k}|^2 +
  |g_{1\bar pk}|^2}{\lambda_1 - \widetilde{\lambda}_p} \\
&\quad + (B^{11})_{k\bar k} - 2\mathrm{Re}\sum_{p > 1} \frac{
  g_{p\bar 1k}(B^{1\bar p})_{\bar k} + g_{1\bar pk}(B^{p\bar 1})_{\bar
    k}}{ \lambda_1 - \widetilde{\lambda}_p} + \widetilde{\lambda}_1^{pq,rs}(B^{pq})_k
(B^{rs})_{\bar k} 
\end{aligned}\]
where we used Equation~\eqref{eq:l1}. Note that $B$ is a constant
matrix in our local coordinates, but its covariant derivatives may not
vanish. The
assumption $\sum_i \lambda_i > 0$ implies that $\sum_i \widetilde{\lambda}_i >
0$ if the matrix $B$ is sufficiently small, and so $|\widetilde{\lambda}_i| <
(n-1)\lambda_1$ for all $i$, which implies $(\lambda_1 -
\widetilde{\lambda}_p)^{-1} \geq (n\lambda_1)^{-1}$. Since we are
trying to bound $\lambda_1$ from above, we can assume $\lambda_1 >
1$.  We can also absorb the terms
$g_{p\bar1 k}(B^{1\bar p})_{\bar k}$ using
\[ \big| g_{p\bar 1k}(B^{1\bar p})_{\bar k}\big| \leq \frac{1}{4}
|g_{p\bar 1k}|^2 + C|B^{nn}|^2, \]
using that the derivative of $B$ is controlled by $B^{nn}$. In addition for $p > 1$ we have
\[ \frac{1}{\lambda_1 - \widetilde{\lambda}_p} \leq \frac{1}{B^{22}} <
\frac{2}{B^{nn}}, \]
and so 
\[ \frac{|B^{nn}|^2}{\lambda_1 - \widetilde{\lambda}_p} +
\widetilde{\lambda}_1^{pq,rs} (B^{pq})_k(B^{rs})_{\bar k} =
O(B^{nn}).\] 
 It follows that
\[ \widetilde{\lambda}_{1,k\bar k} \geq g_{1\bar 1k\bar k} +
\frac{1}{2n\lambda_1}\sum_{p > 1} (|g_{p\bar 1k}|^2 + |g_{1\bar
  pk}|^2) - 1,\]
if $B$ is chosen sufficiently small. 
Using that $g_{i\bar j} = \chi_{i\bar j} + u_{i\bar j}$, we get
\[ \begin{aligned} \label{eq:l2}
\widetilde{\lambda}_{1,k\bar k} &\geq \chi_{1\bar 1k\bar k} + u_{1\bar 1k\bar k} +
\frac{1}{2n\lambda_1}
\sum_{p > 1} (|\chi_{p\bar 1k} + u_{p\bar 1k}|^2 + |\chi_{1\bar pk} + u_{1\bar
    pk}|^2) - 1  \\
&\geq u_{1\bar 1k\bar k} + \frac{1}{3n\lambda_1}\sum_{p > 1}
(|u_{p\bar 1k}|^2 + |u_{1\bar pk}|^2) - C_0,
\end{aligned}\]
where $C_0$ is a constant depending only on the background data
(including $\chi$). From
here on out $C_0$ will always denote such a constant which may vary
from line to line, but does not depend on other parameters that we
choose later on. 

Commuting derivatives, we obtain
\[\label{eq:ff2} u_{1\bar 1k\bar k} = u_{k\bar k 1\bar 1} - 2\mathrm{Re}( u_{k\bar p
  1}\overline{T^p_{k1}}) + u_{i\bar j}\ast R + u_{i\bar j}\ast T\ast T, \]
where $R, T$ are the curvature and torsion of $\alpha$, and $\ast$
denotes a contraction. Using this in \eqref{eq:l2} we get
\[ \widetilde{\lambda}_{1,k\bar k} \geq u_{k\bar k1\bar 1} + \frac{1}{3n\lambda_1}\sum_{p > 1}
(|u_{p\bar 1k}|^2 + |u_{1\bar pk}|^2) -2\mathrm{Re}(u_{k\bar
  p1}\overline{T^p_{k1}}) - C_0\lambda_1, \]
since $u_{i\bar j}$ is controlled by $\lambda_1$, and we assumed
$\lambda_1 > 1$. 

We have  $u_{k\bar p1} = u_{1\bar pk} + u_{i\bar j}\ast T$. This means
that we can absorb almost all of the terms $u_{k\bar
  p1}\overline{T^p_{k1}}$ using the good positive sum, except for
$u_{k\bar 11}\overline{T^1_{k1}}$. We also rewrite $u$ in terms of
$g$, to finally obtain
\[ \label{eq:l3}
\widetilde{\lambda}_{1,k\bar k} \geq g_{k\bar k1\bar 1} - 2\mathrm{Re}(g_{k\bar
  11}\overline{T^1_{k1}}) - C_0\lambda_1. \]

Differentiating the equation $F(A) = h$, we have
\[ h_1 = F^{ij}g_{i\bar j1} = F^{kk}g_{k\bar k1},\]
\[ h_{1\bar 1} = F^{pq, rs} g_{p\bar q1}g_{r\bar s\bar 1} +
F^{kk}g_{k\bar k 1\bar 1}, \]
using that $F^{ij}$ is diagonal at the origin (since $A$ is
diagonal). Using this in Equation~\eqref{eq:l3} we get
\[ F^{kk}\widetilde{\lambda}_{1,k\bar k} \geq -F^{pq,rs}g_{p\bar q1}g_{r\bar s\bar
  1} - 2F^{kk}\mathrm{Re}(g_{k\bar 11}\overline{T^1_{k1}}) - C_0\lambda_1\mathcal{F}, \]
where we wrote $\mathcal{F} = \sum_k F^{kk}$ and used that
$\mathcal{F} > \tau$ to absorb a constant into $\lambda_1\mathcal{F}$.
Defining the linearized operator $Lw = F^{ij}w_{i\bar j}$, we have
\[ \label{eq:l4}
\begin{aligned}
L(\log \widetilde{\lambda}_1) &\geq \frac{ - F^{pq,rs}g_{p\bar q1}g_{r\bar
    s\bar1}}{\lambda_1} - \frac{F^{kk}|g_{1\bar 1k}|^2}{\lambda_1^2}
\\
&\quad -
\frac{F^{kk}}{\lambda_1} 2\mathrm{Re}(g_{k\bar 1 1}\overline{T^1_{k1}})
  - C_0\mathcal{F}.
\end{aligned}
\]
We have
\[ \label{eq:l7} \begin{aligned}
 g_{1\bar 1k} &= \chi_{1\bar 1k} + u_{1\bar 1k} \\
&= \chi_{1\bar 1k} + u_{k\bar 11} - T^1_{k1}\lambda_1 \\
&= (\chi_{1\bar 1k} - \chi_{k\bar11}) + g_{k\bar 11} - T^1_{k1}\lambda_1,
\end{aligned}\]
and so
\[ |g_{1\bar 1k}|^2 \leq |g_{k\bar 11}|^2 -
2\lambda_1\mathrm{Re}(g_{k\bar11}\overline{T^1_{k1}}) + C_0(
\lambda_1^2 + |g_{1\bar 1 k}|)\]

Using this and \eqref{eq:l7} again in Equation~\eqref{eq:l4}, we get
\[ \label{eq:l5}
\begin{aligned}
L(\log \widetilde{\lambda}_1) &\geq \frac{-F^{pq,rs}g_{p\bar q1}g_{r\bar s\bar
    1}}{\lambda_1} - \frac{F^{kk}|g_{k\bar 11}|^2}{\lambda_1^2} - C_0(\mathcal{F} + \lambda_1^{-2}|F^{kk}g_{1\bar 1 k}|). 
\end{aligned} \]

As a reminder, we note that in this calculation $\widetilde{\lambda}_1$ denotes the largest
eigenvalue of the perturbed endomorphism $\widetilde{A} = A - B$. At
the point where we are calculating, this coincides with the largest
eigenvalue of $A$, but at nearby points it is a small perturbation. We
could take $B\to 0$, and obtain the same differential inequality
\eqref{eq:l5}  for the largest eigenvalue of $A$ as well, but this
would only hold in a viscosity sense because the largest eigenvalue of
$A$ may not be $C^2$ at the origin, if some eigenvalues coincide.

We now begin the main calculation for proving
Proposition~\ref{prop:HMW}. 
\begin{proof}[Proof of Proposition~\ref{prop:HMW}]
Set $K = \sup |\nabla u|^2 + 1$, and consider the function
\[ G = \log\widetilde{\lambda}_1 + \phi(|\nabla u|^2) + \psi(u), \]
where $\phi$ is the same as the function used in \cite{HMW10}:
\[ \phi(t) = -\frac{1}{2}\log\left(1 - \frac{t}{2K}\right),\] and it satisfies 
\[ (4K)^{-1} < \phi' < (2K)^{-1}, \quad \phi'' = 2\phi'^2 > 0. \]
We normalize $u$ so that $\inf u=0$, so from Proposition~\ref{prop:C0}
we already have a bound on $\sup u$. We then let $\psi: [0,\sup
u]\to\mathbf{R}$ be defined by
\[ \psi(t) = -2At + \frac{A\tau}{2} t^2, \]
where $\tau$ is chosen sufficiently small depending on $\sup u$ (we decrease the $\tau$ from
before if necessary), so that $\psi$ satisfies the bounds
\[ A \leq -\psi' \leq 2A, \quad \psi'' = A\tau. \]
Here $A$ is a large constant that we will choose later.

Let us write $w_k=-(B^{11})_k / \lambda_1$, which appears in the
derivative of $\log\widetilde{\lambda}_1$. We assume $\lambda_1 > 1$,
so this is a bounded quantity. We have
\[ \begin{aligned}
  G_k &= \frac{g_{1\bar 1k}}{\lambda_1} + \phi'(u_{pk}u_{\bar p} + u_p
  u_{\bar pk}) + \psi' u_k + w_k\\
  G_{kk} &= (\log \widetilde{\lambda}_1)_{k\bar k} + \phi''\Big| u_{pk}u_{\bar p} +
  u_pu_{\bar pk}\Big|^2  \\
  &\quad + \phi'\Big( u_{pk\bar k}u_{\bar p} + u_p u_{\bar pk\bar k} +
  \sum_p(|u_{pk}|^2 + |u_{\bar pk}|^2)\Big) + \psi'' u_ku_{\bar k} +
  \psi' u_{k\bar k}.
  \end{aligned}\]

Commuting derivatives, we have the identities
\[ \begin{aligned}  u_{pk\bar k} &=  u_{k\bar kp} - T^q_{kp}u_{q\bar k} + R_{k\bar
    kp}^{\,\,\,\,\,\, q} u_q \\
  &= u_{k\bar kp} - T^k_{kp}\lambda_k + T^q_{kp}\chi_{q\bar k} +
  R_{k\bar k p}^{\,\,\,\,\,\, q}u_q
  \end{aligned}\]
and
\[ \begin{aligned}
  u_{\bar pk\bar k} &= u_{k\bar k\bar p} - \overline{T^q_{kp}}u_{k\bar
    q} \\
  &= u_{k\bar k\bar p} - \overline{T^k_{kp}}\lambda_k +
  \overline{T^q_{kp}}\chi_{k\bar q}.
  \end{aligned} \]
Differentiating the equation $F(A)=h$ once, we have
\[ F^{kk}u_{k\bar kp} = F^{kk}(g_{k\bar kp} - \chi_{k\bar kp}) = h_p -
F^{kk}\chi_{k\bar kp}, \]
and so
\[ \begin{aligned} F^{kk}u_{pk\bar k}u_{\bar p} &= h_pu_{\bar p} - F^{kk}\chi_{k\bar
  kp}u_{\bar p} - T^k_{kp} F^{kk}\lambda_k u_{\bar p} + T^q_{kp}
F^{kk}\chi_{q\bar k}u_{\bar p} + F^{kk}R_{k\bar kp}^{\,\,\,\,\,\,\,\,\,
  q}u_qu_{\bar p} \\
&\geq - C_0(K^{1/2} + K^{1/2}\mathcal{F} + K^{1/2}\mathcal{F} +
K\mathcal{F}) - \epsilon_1 F^{kk}\lambda_k^2 -
C_{\epsilon_1}\mathcal{F}K \\
&\geq -C_0K\mathcal{F} - \epsilon_1F^{kk}\lambda_k^2 -
C_{\epsilon_1}\mathcal{F}K. 
\end{aligned}\]
We have used the inequality (valid for each $k,p$)
\[ |F^{kk}\lambda_k u_{\bar p}| \leq \epsilon_1 F^{kk}\lambda_k^2 +
C_{\epsilon_1}F^{kk}K, \]
for any $\epsilon_1 > 0$ and corresponding $C_{\epsilon_1}>0$, which
implies that if we sum over $k$ then
\[ |F^{kk}T^k_{kp}\lambda_ku_{\bar p}| \leq \epsilon_1
F^{kk}\lambda_k^2 + C_{\epsilon_1}\mathcal{F}K. \]
The same estimate also holds for $F^{kk}u_{\bar pk\bar k}u_p$
(which has fewer terms). We
have $\phi' < (2K)^{-1}$, so combining these estimates we obtain
\[ \phi'F^{kk}(u_{pk\bar k}u_{\bar p} + u_pu_{\bar pk\bar k}) \geq -
C_0\mathcal{F} - \epsilon_1 K^{-1} F^{kk}\lambda_k^2 -
C_{\epsilon_1}\mathcal{F}. \]
This implies that at the maximum of $G$
\[ \begin{aligned}
  0 \geq LG &\geq L(\log \widetilde{\lambda}_1) + F^{kk}\phi''\Big|u_{pk}u_{\bar p} + u_p
  u_{\bar pk}\Big|^2 + F^{kk}\phi'\sum_p( |u_{pk}|^2 + |u_{\bar
    pk}|^2) \\
  &\quad + \psi'' F^{kk}u_ku_{\bar k} + \psi' F^{kk}u_{k\bar k} - C_0\mathcal{F} \\
&\geq \frac{-F^{ij, rs}g_{i\bar j1}g_{r\bar s\bar 1}}{\lambda_1} - \frac{F^{kk}|g_{k\bar
    11}|^2}{\lambda_1^2}  +F^{kk}\phi''\Big|u_{pk}u_{\bar p} + u_p
  u_{\bar pk}\Big|^2 \\
&\quad + F^{kk}\phi'\sum_p( |u_{pk}|^2 + |u_{\bar
    pk}|^2) + \psi''F^{kk} u_ku_{\bar k} + F^{kk} \psi'u_{k\bar k} \\
&\quad - C_0( 
  \mathcal{F} + \lambda_1^{-2}|F^{kk}g_{1\bar 1k}|) - C_{\epsilon_1}\mathcal{F} -
  \epsilon_1 K^{-1} F^{kk}\lambda_k^2. 
  \end{aligned}\]

We have
\[\begin{aligned} F^{kk}|u_{k\bar k}|^2 &= F^{kk}(\lambda_k -
  \chi_{k\bar k})^2 \\
&\geq \frac{1}{2}F^{kk}\lambda_k^2 - C_0\mathcal{F}. \\
\end{aligned}\]
Note that
$\phi' > (4K)^{-1}$, so if we choose
$\epsilon_1 = \frac{1}{16}$, we can use half of the $\phi' F^{kk}|u_{k\bar k}|^2$ term to
cancel the negative $\epsilon_1$ term. Since this fixes $\epsilon_1$,
we can absorb the $C_{\epsilon_1}$ term into $C_0$. It follows that 
\[ \label{eq:l9} \begin{aligned}
0 &\geq \frac{-F^{ij, rs}g_{i\bar j1}g_{r\bar s\bar 1}}{\lambda_1} - \frac{F^{kk}|g_{k\bar
    11}|^2}{\lambda_1^2}  +F^{kk}\phi''\Big|u_{pk}u_{\bar p} + u_p
  u_{\bar pk}\Big|^2 \\
&\quad + \frac{1}{32K}F^{kk}\lambda_k^2 + \frac{1}{16K}F^{kk}\sum_p( |u_{pk}|^2 + |u_{\bar
    pk}|^2) + \psi''F^{kk} u_ku_{\bar k} \\
&\quad + \psi'F^{kk}u_{k\bar k} - C_0( \mathcal{F} + \lambda_1^{-2}|F^{kk}g_{1\bar 1k}|).
  \end{aligned}\]
To deal with the final term, we use the equation $G_k=0$. This
implies that
\[ \frac{g_{1\bar 1k}}{\lambda_1} = -\phi'(u_{pk}u_{\bar p} +
  u_pu_{\bar pk}) - \psi'u_k -w_k, \]
so
\[ \label{eq:m5} \begin{aligned} 
 \lambda_1^{-2}F^{kk}|g_{1\bar 1k}| &\leq  \frac{1}{2K}
 \lambda_1^{-1}F^{kk}\Big( |u_{pk}| + |u_{\bar pk}|\Big) K^{1/2} + 2A
 \lambda_1^{-1}F^{kk}|u_k| + C_0\mathcal{F}. 
\end{aligned} \]
It is clear that the term involving $|u_{pk}| + |u_{\bar pk}|$ can be
absorbed by the fifth term in \eqref{eq:l9}, up to adding a further
multiple of $\mathcal{F}$. We therefore have
\[ \label{eq:l10} \begin{aligned}
LG &\geq \frac{-F^{ij, rs}g_{i\bar j1}g_{r\bar s\bar 1}}{\lambda_1} - \frac{F^{kk}|g_{k\bar
    11}|^2}{\lambda_1^2}  +F^{kk}\phi''\Big|u_{pk}u_{\bar p} + u_p
u_{\bar pk}\Big|^2 + \\
&\quad  +\frac{1}{20K} F^{kk}\left(\lambda_k^2 + \sum_p
  |u_{pk}|^2\right) \\
&\quad + \psi''F^{kk} u_ku_{\bar k} + \psi'F^{kk}u_{k\bar k} - C_0(\mathcal{F} + A\lambda_1^{-1}F^{kk}|u_k|), 
  \end{aligned}\]

Following Hou-Ma-Wu~\cite{HMW10} we deal with two cases separately, depending on whether
$-\lambda_n > \delta\lambda_1$ or not, for a small $\delta > 0$ to be
chosen later. 

\bigskip
\noindent {\bf Case 1: $-\lambda_n > \delta\lambda_1$}, so in
particular $\lambda_n < 0$. We use the
equation $G_k=0$ to write
\[ \label{eq:m6} \begin{aligned} -\frac{F^{kk}|g_{1\bar 1k}|^2}{\lambda_1^2} &=
- F^{kk}\Big| \phi'(u_{pk}u_{\bar p} + u_p u_{\bar
  pk}) + \psi' u_k + w_k\Big|^2 \\
&\geq - 2\phi'^2 F^{kk}\Big|u_{pk}u_{\bar p} + u_p u_{\bar
  pk}\Big|^2 - 3(2A)^2 \mathcal{F} K - C_0 
\end{aligned}\]

From \eqref{eq:l7} we also have
\[ \frac{|g_{k\bar 1 1}|^2}{\lambda_1^2} \leq \frac{|g_{1\bar
    1k}|^2}{\lambda_1^2} + C_0(1 + \lambda_1^{-1}|g_{1\bar 1k}|), \]
and using the inequality \eqref{eq:m5} absorbing the $|u_{pk}|,
|u_{\bar pk}|$ terms in the same way, we obtain
\[ \label{eq:m7}
\begin{aligned} \frac{F^{kk}|g_{k\bar 1 1}|^2}{\lambda_1^2} &\leq
  \frac{F^{kk}|g_{1\bar 1k}|^2}{\lambda_1^2} + 2AF^{kk}|u_k| +
C_0\mathcal{F} \\
&\leq \frac{F^{kk}|g_{1\bar 1k}|^2}{\lambda_1^2} + 2A\mathcal{F}K^{1/2} + C_0\mathcal{F}
\end{aligned}
\]
Using this and \eqref{eq:m6}
in our inequality \eqref{eq:l10} for $LG$, together with concavity of $F$
and $\phi'' = 2\phi'^2$,  we find that at the origin 
\[\label{eq:a1}\begin{aligned} 
 0 &\geq \frac{1}{32K}F^{kk}\lambda_k^2 + \psi' F^{kk}u_{k\bar k} -
 C_0(\mathcal{F} + A\lambda_1^{-1}\mathcal{F}K^{1/2}) -
 2A\mathcal{F}K^{1/2} - 12A^2\mathcal{F}K \\
&\geq \frac{1}{32K}F^{kk}\lambda_k^2 + \psi'F^{kk}u_{k\bar k} -
C_0\mathcal{F} - 14A^2\mathcal{F}K,
\end{aligned} \]
where we have assumed that $\lambda_1 > C_0$ and $A > 1$. Note that
$F^{nn} \geq \frac{1}{n}\mathcal{F}$, and  $\lambda_n^2 >
\delta^2\lambda_1$ by our assumption. In addition 
\[ \begin{aligned}
 \psi' F^{kk}u_{k\bar k} &= \psi'F^{kk}g_{k\bar k} - \phi'F^{kk}\chi_{k\bar
   k} \\
&\geq - 2A\mathcal{F}\lambda_1 - C_0\mathcal{F}, 
\end{aligned}\]
and so from \eqref{eq:a1} we have
\[ \begin{aligned} 
 0 &\geq \frac{1}{32nK}\mathcal{F}\delta^2\lambda_1^2 -
 2A\mathcal{F}\lambda_1 - C_0\mathcal{F} - 14A^2\mathcal{F}K.
\end{aligned} \]
Dividing by $K\mathcal{F}$ and rearranging we get the required bound
for $K^{-1}\lambda_1$, with the bound depending on $\delta, A$ that
will be chosen in Case 2 below. 

\bigskip
\noindent {\bf Case 2:} $-\lambda_n \leq \delta \lambda_1$. Define the
set
\[ I = \{i\,:\, F^{ii} > \delta^{-1}F^{11}\}. \]
Using that $G_k=0$ as above, we have
\[ \begin{aligned}
 -\sum_{k\not\in I} \frac{ F^{kk}|g_{1\bar
    1k}|^2}{\lambda_1^2} &= - 2\phi'^2\sum_{k\not\in I} F^{kk}\Big|
u_{pk}u_{\bar p} + u_p u_{\bar pk}\Big|^2 -
3\psi'^2\sum_{k\not\in I} F^{kk}|u_k|^2 -C_0 \\
&\geq -\phi'' \sum_{k\not\in I} F^{kk}\Big|u_{pk}u_{\bar p} + u_p
u_{\bar pk}\Big|^2 - 3(2A)^2 \delta^{-1}F^{11}K - C_0.
\end{aligned}\]
In addition, for $k\not\in I$, from \eqref{eq:m7} we have
\[ \frac{F^{kk}|g_{k\bar 11}|^2}{\lambda_1^2} \leq
\frac{F^{kk}|g_{1\bar 1k}|^2}{\lambda_1^2} +
2A\delta^{-1}F^{11}K^{1/2} + C_0\mathcal{F}.\]
Our inequality \eqref{eq:l10} for $LG$ then implies
\[ \label{eq:11}\begin{aligned}
 0 &\geq \frac{-F^{ij,rs}g_{i\bar j1}g_{r\bar s\bar 1}}{\lambda_1} -
\sum_{k\in I} \frac{F^{kk}|g_{k\bar 11}|^2}{\lambda_1^2} +
 \phi''\sum_{k\in I} F^{kk}\Big| u_{pk}u_{\bar p} +u_pu_{\bar
   pk}\Big|^2  \\
&\quad + \psi''F^{kk}|u_k|^2 + \frac{1}{32K}F^{kk}\lambda_k^2 +
\psi' F^{kk}u_{k\bar k} \\
&\quad - C_0(\mathcal{F} + A\lambda_1^{-1}F^{kk}|u_k|) - 13A^2\delta^{-1}F^{11}K. \\
\end{aligned}\]

We want to choose $\delta$ so small that
\[ \label{eq:20}\frac{4\psi'^2\delta}{1-\delta} \leq \frac{1}{2}\psi''. \]
Note that $|\psi'| \leq 2A$, and $\psi'' = \tau A$ for a fixed $\tau >
0$, so we can choose $\delta \leq \delta_0A^{-1}$, for some fixed number
$\delta_0$ (depending on $\tau$). 

To deal with the first four terms in \eqref{eq:11} we use that $G_k=0$
to obtain
\[\label{eq:21} \begin{aligned} 2\phi'^2 \sum_{k\in I} & F^{kk}\Big|u_{pk}u_{\bar p} + u_pu_{\bar
  pk}\Big|^2 = 2\sum_{k\in I} F^{kk}\left| \frac{g_{1\bar
      1k}}{\lambda_1} + \psi'u_k + w_k\right|^2 \\
 & \geq 2\delta\sum_{k\in I} \frac{F^{kk}|g_{1\bar1 k}|^2}{\lambda_1^2}
  - \frac{4\delta\psi'^2}{1-\delta} \sum_{k\in I} F^{kk}|u_k|^2 - C_0, 
\end{aligned}\]
just as in \cite{HMW10}, using the elementary inequality $|a+b|^2 \geq
\delta |a|^2 - \frac{\delta}{1-\delta}|b|^2$. More precisely we used
\[ |a + b + c|^2 \geq \delta |a|^2 - \frac{\delta}{1-\delta}|b+c|^2
\geq \delta|a|^2 - \frac{2\delta}{1-\delta}|b|^2 -
\frac{2\delta}{1-\delta}|c|^2. \]
Using \eqref{eq:m7} we have
\[ \label{eq:f1} \begin{aligned}
 2\delta\sum_{k\in I} \frac{F^{kk}|g_{1\bar 1k}|^2}{\lambda_1^2} &\geq
 2\delta\sum_{k\in I} \frac{F^{kk}|g_{k\bar 11}|^2}{\lambda_1^2}
 -4\delta AF^{kk}|u_k|^2 - \delta C_0\mathcal{F} \\
&\geq 2\delta\sum_{k\in I} \frac{F^{kk}|g_{k\bar 11}|^2}{\lambda_1^2}
- \delta^2A^2F^{kk}|u_k|^2 - C_0\mathcal{F}. 
\end{aligned} \]
We also have
\[ A\lambda_1^{-1}F^{kk}|u_k| \leq \frac{1}{2}
A^2\lambda_1^{-1}F^{kk}|u_k|^2 + \frac{1}{2}\lambda_1^{-1}\mathcal{F}. \]
Using this, \eqref{eq:f1} and \eqref{eq:21} in \eqref{eq:11} we have
\[ \label{eq:f2} \begin{aligned}
 0 &\geq \frac{-F^{ij,rs}g_{i\bar j1}g_{r\bar s\bar 1}}{\lambda_1} -
(1-2\delta) \sum_{k\in I} \frac{F^{kk}|g_{k\bar 11}|^2}{\lambda_1^2} \\
&\quad +\left( \frac{1}{2} \psi'' - \delta^2A^2 -
  A^2\lambda_1^{-1}\right) F^{kk}|u_k|^2 + \frac{1}{32K}F^{kk}\lambda_k^2 +
\psi' F^{kk}u_{k\bar k} \\
&\quad - C_0\mathcal{F} - 13A^2\delta^{-1}F^{11}K. \\
\end{aligned}\]

In addition we claim that 
\[ \label{eq:con2}
\frac{-F^{ij,rs} g_{i\bar j1}g_{r\bar s\bar1}}{\lambda_1} -
(1-2\delta)\sum_{k\in I} \frac{F^{kk}|g_{k\bar11}|^2}{\lambda_1^2}
\geq 0. \]
Indeed, from Equation \eqref{eq:concaveF} we have
\[ F^{ij,rs}g_{i\bar j1}g_{r\bar s\bar 1} \leq \sum_{k\in I}
\frac{F^{11}-F^{kk}}{\lambda_1 - \lambda_k} |g_{k\bar 11}|^2, \]
using that $\frac{F^{11}-F^{kk}}{\lambda_1 - \lambda_k}\leq 0$. For
$k\in I$ we have $F^{11} < \delta F^{kk}$, and so
\[ \frac{F^{11} - F^{kk}}{\lambda_1 - \lambda_k} \leq 
\frac{(\delta-1)F^{kk}}{\lambda_1-\lambda_k}. \]
It is therefore enough to show
\[ \frac{\delta - 1}{\lambda_1 - \lambda_k} \leq
-\frac{1-2\delta}{\lambda_1}. \]
Rearranging, this is equivalent to $(2\delta - 1)\lambda_k \leq
\delta\lambda_1$. If $\lambda_k \geq 0$, this is clear, while if
$\lambda_k < 0$, then 
\[ (2\delta - 1)\lambda_k \leq -\lambda_k \leq -\lambda_n \leq
\delta\lambda_1, \]
where we used our assumption for Case 2. 

Using \eqref{eq:con2} in \eqref{eq:f2} we have
\[ \label{eq:f3} \begin{aligned}
 0 &\geq \left( \frac{1}{2} \psi'' - \delta^2A^2 -
  A^2\lambda_1^{-1}\right) F^{kk}|u_k|^2 + \frac{1}{32K}F^{kk}\lambda_k^2 +
\psi' F^{kk}u_{k\bar k} \\
&\quad - C_0\mathcal{F} - 13A^2\delta^{-1}F^{11}K. \\
\end{aligned}\]

Suppose that $\lambda_1 > R$ with the $R$ in
\eqref{eq:subsol3}. There are two cases to consider:
\begin{itemize}
\item If \eqref{eq:subsol2} holds, then we have $\psi' F^{kk}u_{k\bar
    k} > A\kappa \mathcal{F}$. Choosing $A$ so that $A\kappa > C_0$,
  Equation \eqref{eq:f3} implies
\[ \begin{aligned}
 0 &\geq \left( \frac{1}{2} \psi'' - \delta^2A^2 -
  A^2\lambda_1^{-1}\right) F^{kk}|u_k|^2 +
\frac{1}{32K}F^{kk}\lambda_k^2 
- 13A^2\delta^{-1}F^{11}K. \\
\end{aligned}\]
At this point $\psi'' = \tau A$ is fixed, so if we choose $\delta$
sufficiently small, and $\lambda_1$ is sufficiently large, then then
first term is positive. We then obtain
\[ 0 \geq \frac{1}{32K}F^{11}\lambda_1^2 
- 13A^2\delta^{-1}F^{11}K, \]
from which the required bound for $K^{-1}\lambda_1$ follows. 
\item If \eqref{eq:subsol2} does not hold, then we must have $F^{11} >
  \kappa\mathcal{F}$.  We have
  \[ \begin{aligned}
    F^{kk}u_{k\bar k} &= F^{kk}\lambda_k - F^{kk}\chi_{k\bar k} \\
&\leq \frac{1}{128AK}F^{kk}\lambda_k^2 +
  C_0AK\mathcal{F}, 
\end{aligned} \]
which together with \eqref{eq:f3} and $|\psi'| < 2A$ implies 
\[ 0 \geq \frac{1}{64K}\kappa \mathcal{F}\lambda_1^2 - C_0A^2K\mathcal{F} -
13A^2\delta^{-1} \mathcal{F}K, \]
where we also used the bounds $\kappa\mathcal{F} < F^{11} <
\mathcal{F}$. This inequality again implies the bound of the form
$\lambda_1 < CK$, which we are after. 
\end{itemize} 
\end{proof}

\begin{remark}\label{rem:C0C2}
  Under an extra concavity condition for $f$, one can obtain
  $C^2$-estimates directly from $C^0$-estimates, as is well-known in
  the case of the Monge-Amp\`ere equation. See
  e.g. Yau~\cite{Yau78}, Tosatti-Weinkove~\cite{TW10_1} and also Weinkove~\cite{W06},
  Fang-Lai-Ma~\cite{FLM11}, Sun~\cite{Sun13} for
  the inverse $\sigma_k$-equations.

  Suppose that in addition to
  the structural conditions (i), (ii), (iii) in the introduction, $f$ 
  satisfies the extra conditions 
\[ \label{eq:strongconv}
f_{11} + \frac{f_1}{\lambda_1} \leq 0,\quad \text{and }\lambda_1f_1 \leq
\lambda_if_i\quad \text{ for all }i > 1, \]
where $\lambda_1$ is the largest eigenvalue, and we work at
a diagonal matrix $A$, as before. See Fang-Lai~\cite{FL12_1} for
related conditions. In particular it is easy to check that these conditions hold for the
complex Monge-Amp\`ere equation written in the form $f(\lambda) = \log
\lambda_1\cdot\ldots \lambda_n$. 

These conditions imply a stronger
concavity property of $F$, which allows us to get rid of the first
two terms in \eqref{eq:l5}. Indeed, from \eqref{eq:concaveF} we have
\[ \label{eq:concave_form}
F^{pq,rs}g_{p\bar q1}g_{r\bar s\bar1} \leq f_{kk}|g_{k\bar k1}|^2
 + \sum_{k > 1} \frac{f_1 - f_k}{\lambda_1 -
   \lambda_k}|g_{k\bar 11}|^2. \]
For $k > 1$ the condition $\lambda_1f_1 \leq \lambda_kf_k$ implies
\[ -\frac{1}{\lambda_1} \frac{f_1 - f_k}{\lambda_1 - \lambda_k} -
\frac{f_k}{\lambda_1^2} \geq 0.\]
From this we have
\[ \begin{aligned}
-\frac{F^{pq,rs}g_{p\bar q1}g_{r\bar s\bar 1}}{\lambda_1} - \frac{F^{kk}|g_{k\bar
    11}|^2}{\lambda_1^2} \geq -\frac{f_{11}|g_{1\bar 11}|^2}{\lambda_1} -
\frac{F^{11}|g_{1\bar 11}|^2}{\lambda_1^2} \geq 0,
\end{aligned}\]
using the assumption $\lambda_1 f_{11} + f_1 \leq 0$ as well. Using
\eqref{eq:l5} we have
\[ \label{eq:l6} L(\log \widetilde{\lambda}_1) \geq - C_0(\mathcal{F} +
\lambda_1^{-2}F^{kk}|g_{1\bar1 k}|). \]

Note that \eqref{eq:strongconv} implies that $\lambda_i > 0$ for all
$i$, and so in particular we must have $\Gamma=\Gamma_n$. Suppose that
$\sigma > \sup_{\d \Gamma} f$ and $\d\Gamma^\sigma$ is non-empty, and
let $x_0\in \d\Gamma^\sigma$ be the closest point to the origin. For
any $\epsilon > 0$ we can find an $R > 0$ such that if $\lambda\in
\d\Gamma^\sigma$ and $\lambda_1 > R$, then $f_1 < \epsilon
\mathcal{F}$ at $\lambda$. To see
this, let $\mathbf{n}_\lambda$ be the inward pointing unit normal to
$\d\Gamma^\sigma$ at $\lambda$. By convexity we have
\[ \mathbf{n}_\lambda\cdot(\lambda - x_0) < 0, \]
which implies $\mathbf{n}_\lambda\cdot \lambda <
\mathbf{n}_\lambda\cdot x_0$. In particular we have 
\[ \lambda_1 \frac{f_1(\lambda)}{|\nabla f(\lambda)|} < |x_0|. \]
In view of the equivalence between $|\nabla f|$ and $\mathcal{F}$ (see
\eqref{eq:s2}), this implies our claim. In particular in
Proposition~\ref{prop:subsol1} only the alternative \eqref{eq:s3}
occurs. It follows that if the largest eigenvalue $\lambda_1$ is
sufficiently large, then Equation~\eqref{eq:subsol2} will be
satisfied, if we assume the existence of a $\mathcal{C}$-subsolution. 

Recall that $g_{i\bar j} = \chi_{i\bar j} + u_{i\bar j}$. Let us
normalize $u$ so that
$\inf_M u = 0$, and consider the function
\[ G = \log(\lambda_1) + \psi(u), \]
where $\psi : [0,\infty)\to\mathbf{R}$ is defined by
\[ \psi(x) = -Ax + \frac{1}{1+x} -1, \]
for a constant $A > 1$ to be chosen later (this is the function used
by Phong-Sturm~\cite{PS09}). Then for $x\geq 0$ we have
\[ \begin{gathered}
-Ax - 1 < \psi \leq -Ax, \\
 A < -\psi' < 2A, \\
\psi'' = \frac{2}{(1+x)^3}.
\end{gathered}\]
At the maximum
of $G$ we compute in coordinates as before, i.e. $\alpha$ is the
identity and $g$ is diagonal, and we make a small perturbation to the
matrix $\alpha^{i\bar p}g_{j\bar p}$ by a diagonal matrix, so that the eigenvalues are distinct
at this point (this does not change $\lambda_1$ at the
origin). Writing $w_k = -(B^{11})_k / \lambda_1$ as before, at the origin we have
\[ \begin{aligned} 
0 &= G_k = \frac{g_{1\bar 1k}}{\lambda_1} + \psi' u_k + w_k\\
0 &\geq L(G) = L(\log \widetilde{\lambda}_1) + \psi'' F^{kk}u_k u_{\bar
  k} + F^{kk}\psi' u_{k\bar k}, 
\end{aligned}\] 
and so using \eqref{eq:l6}
\[ \begin{aligned}
 0 &\geq \psi''F^{kk}u_ku_{\bar k} + \psi' F^{kk}(g_{k\bar
  k}-\chi_{k\bar k}) - C_0( \mathcal{F} +
\lambda_1^{-2}F^{kk}|g_{1\bar 1k}|). 
\end{aligned}
\]
Using the properties of $\psi$, if we assume that $\lambda_1 > R$ for
sufficiently large $R$, we will have
\[ 0 \geq 2(1+u(0))^{-3} F^{kk}|u_k|^2 + A\kappa\mathcal{F} -
C_0\mathcal{F}
- C_0\lambda_1^{-2}F^{kk}|g_{1\bar 1k}|.\]
Choose $A$ so that $A\kappa = C_0 + 1$. This implies
\[ \label{eq:l8} 0 \geq 2(1+u(0))^{-3} F^{kk}|u_k|^2 + \mathcal{F} -
C_0\lambda_1^{-2}F^{kk}|g_{1\bar 1k}|. \]
Using the equation $G_k=0$ we have (we can assume that $|w_k| < 1$)
\[ \begin{aligned}
  C_0\lambda_1^{-2}F^{kk}|g_{1\bar1k}| &=
  C_0\lambda_1^{-1}F^{kk}(|\psi'u_k| + 1) \\
  &\leq
\lambda_1^{-1}F^{kk}|u_k|^2 +
\big[(2C_0A)^2+C_0\big]\lambda_1^{-1}\mathcal{F}.
\end{aligned}\]

Using this in \eqref{eq:l8}, we have 
\[ 0 \geq \Big(2(1+u(0))^{-3} - \lambda_1^{-1}\Big)F^{kk}|u_k|^2 + \Big(1
- \big[(2C_0A)^2+C_0\big]\lambda_1^{-1}\Big)\mathcal{F}. \]
It follows that at the maximum of $G$ we must have
\[ \lambda_1 < \frac{1}{2}(1+u(0))^3 + C_1,\]
where $C_1$ only depends on the background data (we also choose $C_1 >
R$).  Therefore
\[ G(0) < 3\log(1 + u(0)) + C_2 - Au(0). \]
Since $\log(1 + x) \leq x$, we obtain
\[ G(0) < C_2 - (A-3)u(0) \leq C_2. \]
Since $G$ achieves its maximum at $0$, we have $G(p) \leq C_2$ for all
$p$, from which we finally obtain the global
estimate
\[ \label{eq:trg} \mathrm{tr}_\alpha g \leq Ce^{A(u - \inf_M u)}. \]
Note that we assumed $\inf_M u = 0$ earlier but wrote this last
estimate in the usual form that is independent of the normalization of
$u$. At this point we can either use the $C^0$-estimate
Proposition~\ref{prop:C0} to obtain an estimate for
$\mathrm{tr}_\alpha g$, or this $C^0$-estimate can be directly
derived from \eqref{eq:trg} as in Weinkove~\cite{W06} or
Tosatti-Weinkove~\cite{TW10}.  
\end{remark}

\section{Liouville theorem}\label{sec:DK}

Suppose as before, that $\Gamma\subset \mathbf{R}^n$ is an open convex cone,
containing the positive orthant $\Gamma_n$ and not equal to all of
$\mathbf{R}^n$. In addition assume that $\Gamma$ is preserved under
permuting the coordinates. It follows that
\begin{equation}\label{eq:G}
  \Gamma\subset \{ (x_1,\ldots,x_n)\,:\, \sum x_i > 0\}.
\end{equation}

\begin{defn}  \label{defn:Gammasol}
 Suppose $u: \mathbf{C}^n\to\mathbf{R}$ is continuous. We say that $u$ is
 a (viscosity) $\Gamma$-subsolution if for all $h\in C^2$ such that
 $u-h$ has a local maximum at $z$, we have $\lambda(h_{i\bar j})\in\overline{\Gamma}$,
 where $\lambda(A)$ denotes the eigenvalues of the Hermitian matrix
 $A$.
  
  We say that $u$ is a $\Gamma$-solution, if it is a
  $\Gamma$-subsolution and in addition for all $z\in \mathbf{C}^n$, if
  $h\in C^2$ and $u-h$ has a local minimum at $z$, then
  $\lambda(h_{i\bar j}(z))\in \mathbf{R}^n\setminus
    \Gamma$. 
  \end{defn}

Note that \eqref{eq:G} implies that every $\Gamma$-subsolution is
subharmonic. 
Suppose that we define the function $F_0$ on the space of Hermitian
matrices by the property that
\[ \label{eq:Fdef} \lambda(A) - F_0(A)(1,1,\ldots,1) \in
\overline{\Gamma}, \]
and define $F$ on the space of symmetric $2n\times 2n$ matrices by
$F(M) = F_0(p(M))$, where
\[ p(M) = \frac{M + J^TMJ}{2}, \]
and $J$ is the standard complex structure. 
Then a $\Gamma$-subsolution (resp. solution) $u$ is the same as a viscosity
subsolution (resp. solution) of the nonlinear equation $F(D^2u) =
0$. Note that $F$ is concave and elliptic, but in general not
uniformly elliptic.
Many of the basic results about viscosity subsolutions and
solutions found in Caffarelli-Cabr\'e~\cite{CC95} can still be
applied with the same proofs. In particular we have the following.

\begin{prop}\label{prop:viscprop}
  \begin{enumerate}
    \item If $u_k$ are $\Gamma$-solutions (resp. subsolutions) converging locally uniformly
      to $u$, then $u$ is also a $\Gamma$-solution (resp. subsolution). 
    \item If $u, v$ are $\Gamma$-subsolutions, then $\frac{1}{2}(u+v)$ is also a
      $\Gamma$-subsolution, using that $\Gamma$ is convex. 
   \end{enumerate}
\end{prop}

An important consequence is that mollifications of 
$\Gamma$-subsolutions are again $\Gamma$-subsolutions. Indeed, a
mollification can be written as a uniform limit of averages of a
larger and larger number of translates. 

We will use the following simple comparison result. 
\begin{lem}\label{lem:comp}
  Suppose that $u$ is a smooth $\Gamma$-subsolution and $v$ is a 
  $\Gamma$-solution on a bounded open set $\Omega\subset \mathbf{C}^n$, and in
  addition $u=v$ on $\partial \Omega$. Then $u\leq v$ in $\Omega$. 
\end{lem}
\begin{proof}
  If $v < u$ at some point in $\Omega$, then $v - u$ achieves
  a negative minimum at a point in $\Omega$, so for small $\epsilon >
  0$, the function $v - u - \epsilon|z|^2$ also has a minimum, at a point
  $p\in \Omega$. Since $v$ is a $\Gamma$-solution and $u$
  is smooth, this implies that
  \[ \lambda(u_{i\bar j} + \epsilon \delta_{i\bar j}) \in
  \mathbf{R}^n\setminus \Gamma. \]
  This contradicts that $u$ is a $\Gamma$-subsolution. Indeed, we
  have $\lambda(u_{i\bar  j})\in\overline{\Gamma}$, and since
  $\Gamma_n\subset \Gamma$ and $\Gamma$ is convex, this implies $\lambda(u_{i\bar j} +
  \epsilon \delta_{i\bar j})\in \Gamma$.   
\end{proof}

Note that since the $F$ defined by \eqref{eq:Fdef} is not uniformly elliptic,
the comparison result might not hold in full
generality if $u$ is only a continuous $\Gamma$-subsolution. The
following lemmas will be used in an inductive argument. 

\begin{lem}\label{lem:gamma'}
  Suppose that $\Gamma\not = \Gamma_n$. Then
  $\Gamma'\subset\mathbf{R}^{n-1}$ given by 
\[ \Gamma' = \{(x_1,\ldots,x_{n-1})\,:\, (x_1,\ldots,x_{n-1},0)\in
\Gamma\} \]
  satisfies the same conditions as $\Gamma$. I.e. $\Gamma'$ is a symmetric, open
  convex cone, containing $\Gamma_{n-1}$,
  and $\Gamma'\not=\mathbf{R}^{n-1}$. In addition
  $\overline{\Gamma}\cap \{x_n=0\} = \overline{\Gamma'}$.
\end{lem}
\begin{proof}
  It is clear that $\Gamma'$ is a symmetric open cone, and
  $\Gamma'\ne\mathbf{R}^{n-1}$. 
  The assumption $\Gamma\ne \Gamma_n$, and openness of $\Gamma$,
  means that there is at least one vector in $\Gamma$ with a negative
  entry, and in particular $\Gamma'$ is non-empty. 
 Using that $\Gamma_n\subset \Gamma$ and scaling, we can then obtain that
  for a small $\epsilon > 0$, we have $\mathbf{e} = (1, 1, 
  \ldots, 1, -\epsilon)\in \Gamma$.

  If $(x_1,\ldots,x_{n-1})\in \Gamma_{n-1}$, then $(x_1 + 1, \ldots,
  x_{n-1}+1, \epsilon) \in \Gamma$, and adding $\mathbf{e}$ to this vector we
  have $(x_1,\ldots, x_{n-1},0)\in \Gamma$. This implies
  $(x_1,\ldots,x_{n-1})\in \Gamma'$.

  It remains to show the claim about $\overline{\Gamma'}$. The
  inclusion $\overline{\Gamma'}\subset \overline{\Gamma}\cap
  \{x_n=0\}$ is clear (we are thinking of $\Gamma'$ as a subset of the
  hyperplane $\{x_n=0\}$ in $\mathbf{R}^n$).  For the reverse
  inclusion, suppose that $x=(x_1,\ldots, x_{n-1},0)\in
  \overline{\Gamma}\cap \{x_n=0\}$. This implies that we have $y^{(k)}
  = (y^{(k)}_1,\ldots, y^{(k)}_n)\in \Gamma$ converging to $x$. If
  $y^{(k_i)}_n \leq 0$ along a subsequence, then
  \[ y - y^{(k_i)}_n(1,1,\ldots,1) \in \Gamma \cap\{x_n=0\}, \]
  also converges to $x$, and so $(x_1,\ldots,x_{n-1})\in
  \overline{\Gamma'}$. Otherwise $y^{(k_i)}_n > 0$ along a
  subsequence, in which case
  \[ y + \epsilon^{-1}y^{(k_i)}_n\mathbf{e} \in \Gamma\]
  converges to $x$, and again
  $(x_1,\ldots,x_{n-1})\in\overline{\Gamma'}$. 
\end{proof}

\begin{lem}\label{lem:gamma'2}
Suppose that $v:\mathbf{C}^n\to\mathbf{R}$ is a $\Gamma$-solution,
$\Gamma\ne\Gamma_n$, and $v$ is independent of
the variable $z_n$. Then letting $\Gamma' = \Gamma \cap \{x_n=0\}
\subset \mathbf{R}^{n-1}$, the function $w(z_1,\ldots,z_{n-1}) =
v(z_1,\ldots, z_{n-1},0)$ is a $\Gamma'$-solution on
$\mathbf{C}^{n-1}$. 
\end{lem}
\begin{proof}
  Suppose that $h$ is smooth and $w-h$ has a local maximum at a point
  $z=(z_1,\ldots,z_{n-1})$. Then $v - H$ has a local maximum at $Z =
  (z_1,\ldots, z_{n-1},0)$, where $H(z_1,\ldots,z_n) =
  h(z_1,\ldots,z_{n-1})$. Since $v$ is a $\Gamma$-subsolution, we have
  $\lambda(H_{i\bar j}(Z))\in \overline{\Gamma}$, and one eigenvalue
  is zero. Using
  Lemma~\ref{lem:gamma'} this  implies that
  $\lambda(h_{i\bar j}(z))\in \overline{\Gamma'}$, so  $w$ is a
  $\Gamma'$-subsolution. 

  Similarly if $h$ is smooth and $w-h$ has a local minimum at $z$,
  then $v-H$ has a local minimum at $Z$, which implies
  $\lambda(H_{i\bar j}(z))\in \mathbf{R}^n\setminus\Gamma$, and one
  eigenvalue is zero. So $\lambda(h_{i\bar j}(z))\in
  \mathbf{R}^{n-1}\setminus \Gamma'$. 
\end{proof}

Our goal in this section is the following result, generalizing the
Liouville theorem of Dinew-Kolodziej~\cite{DK12} (see also the
analogous result of Tosatti-Weinkove~\cite{TW13_1} for
$(n-1)$-plurisubharmonic functions). The proof follows their arguments
closely. 
\begin{thm}\label{thm:DK}
  Let $u:\mathbf{C}^n\to\mathbf{R}$ be a Lipschitz $\Gamma$-solution
  such that $|u| < C$ and $u$ has Lipschitz constant bounded by $C$. 
Then $u$ is constant. 
\end{thm}

\begin{proof} We use induction over $n$. If $n=1$, then $u$ is
  harmonic, while if $\Gamma =
  \Gamma_n$, then $u$ is plurisubharmonic. In both cases the result
  follows from the fact that a bounded subharmonic function on
  $\mathbf{C}$ is constant. We therefore assume that $n >
  1$ and $\Gamma \ne \Gamma_n$.

  Suppose that $u$ is non-constant, $|\nabla u| < c_0$, and $\inf u =
  0, \sup u=1$. For any function $v$ on $\mathbf{C}^n$, let
  \[ [v]_r(z) = \int_{\mathbf{C}^n} v(z+ rz') \eta(z') \beta^n(z'), \]
  where $\beta = \sum_i dz_i\wedge d\bar z_i$ and
  $\eta:\mathbf{C}^n\to\mathbf{R}$ is a smooth mollifier satisfying $\eta
  >  0$ in $B(0,1)$, $\eta = 0$ outside $B(0,1)$ and
  $\int_{\mathbf{C}^n} \eta\beta^n = 1$. We do this instead of taking
  averages over balls in order to obtain a smooth function. This is
  used in the comparison result Lemma~\ref{lem:comp}. As in
  \cite{DK12}, Cartan's Lemma implies that
  \[ \lim_{r\to\infty} [u^2]_r(z) = \lim_{r\to\infty} [u]_r(z) = 1, \]
  using that $u$ and $u^2$ are subharmonic.
  
  It will be helpful to regularize $u$ slightly, letting
  $u^\epsilon=[u]_\epsilon$ for $\epsilon > 0$. 
  Just as in \cite{DK12}, there are two cases to consider.

\bigskip 
\noindent{\bf Case 1.} In this case we assume that there is a $\rho >
0$, and sequences $\epsilon_k \to 0$, $x_k\in \mathbf{C}^n$, $r_k
\to\infty$ and unit vectors $\xi_k$ (of type $(1,0)$), such that
\[ \label{eq:case1}
[u^2]_{r_k}(x_k) + [u]_\rho(x_k) - 2u(x_k)\geq 4/3, \]
and
\[ \lim_{k\to\infty} \int_{B(x_k, r_k)} |\d_{\xi_k}u^{\epsilon_k}|^2\, \beta^n  =
  0.\]

In this case, translating and rotating $u$ to make $x_k$ the origin,
and $\d_{z^n} = \d_{\xi_k}$, we obtain a sequence of
$\Gamma$-solutions $u_k$, such that
\[ \begin{aligned} \,
  [u_k^2]_{r_k}(0) + [u_k]_\rho(0) - 2u_k(0) &\geq 4/3, \\
  \lim_{k\to\infty} \int_{B(0,r_k)} |\d_{z^n} u_k^{\epsilon_k}|^2\,\beta^n &= 0.
  \end{aligned}\]
The uniform Lipschitz bound implies that we can replace $u_k$ by a
subsequence, converging locally uniformly to $v:\mathbf{C}^n\to 
\mathbf{R}$, which by Lemma~\ref{prop:viscprop} is also a 
$\Gamma$-solution with the same Lipschitz bound.
 In addition we also have that $u_k^{\epsilon_k}\to
v$ locally uniformly, since $\epsilon_k\to 0$. Just as in \cite{DK12}, we find that $v$ is
independent of $z_n$, and so we can define a function $w :
\mathbf{C}^{n-1}\to\mathbf{R}$ by $w(z_1,\ldots, z_{n-1}) =
v(z_1,\ldots, z_{n-1},0)$, and by Lemma~\ref{lem:gamma'2}, $w$ is a
$\Gamma'$-solution with $\Gamma' = \Gamma\cap
\{x_n=0\}$. The induction hypothesis implies that $w$ is constant, and
so $v$ is constant, but
this contradicts \eqref{eq:case1}, using that $0\leq u\leq 1$. 

\bigskip
\noindent {\bf Case 2.} In this case, the assumption in Case 1 does
not hold, so for all $\rho > 0$, there is a constant $C_\rho > 0$ such
that if $\epsilon < C_\rho^{-1}$, $r > C_\rho$, $x\in \mathbf{C}^n$
and $\xi$ is a unit vector, we have
\[\label{eq:derlower}
 \int_{B(x,r)} |\d_\xi u^\epsilon|^2\, dz \geq
C_\rho^{-1}, \]
as long as 
\[ [u^2]_r(x) + [u]_\rho(x) - 2u(x) \geq 4/3. \]

We choose our origin so that $u(0) < 1/9$, and fix $\rho > 0$ such that
$[u]_\rho(0) > 3/4$. Then choose $r > C_\rho$ such that $[u^2]_r(0) >
3/4$ as well. Define the set
\[ U = \{ z\,:\, 2u(z) < [u^2]_r(z) + [u]_\rho(z) - 4/3\}, \]
so that $0\in U$.

\bigskip
\noindent{\bf Claim:} There is a constant $c > 0$ such that $[(u^\epsilon)^2]_r -
c|z|^2$ is a $\Gamma$-subsolution on $U$ for all $\epsilon < C_\rho^{-1}$.
We have
\[ (u^\epsilon)^2_{i\bar j} = 2u^{\epsilon} u^\epsilon_{i\bar j} +
2u^\epsilon_i u^\epsilon_{\bar j},\]
and so
\[ 
\begin{aligned} 
 \left[(u^\epsilon)^2\right]_{r, i\bar j}(z) &= \int_{B(0,1)}
  (u^\epsilon)^2_{i\bar j}(z + rz')\,\eta(z')\,\beta^n(z') \\
&=
  \int_{B(0,1)} 2u^\epsilon u^\epsilon_{i\bar j}(z +
  rz')\,\eta(z')\,\beta^n(z') \\
&\quad + \int_{B(0,1)} 2u^\epsilon_i u^\epsilon_{\bar j}(z
+ rz')\,\eta(z')\,\beta^n(z') \\
&= A_{i\bar j} + B_{i\bar j}.
\end{aligned}
\]
The subsolution
property of $u^\epsilon$ together with the fact that $\Gamma$ is
convex (i.e. the equation $F(D^2u)=0$ is concave) implies that the
matrix $A$ satisfies $\lambda(A_{i\bar j})\in \overline{\Gamma}$. Using
\eqref{eq:derlower}, we find that  $B_{i\bar j} \geq c_2 \beta$ for
some $c_2 > 0$, and can choose $c$ such that $(c |z|^2)_{i\bar j} \leq
c_2 \beta$. It then follows that
\[ \Big(\left[ (u^\epsilon)^2\right]_r - c |z|^2 \Big)_{i\bar j} \geq
A_{i\bar j}, \]
in the sense of inequalities for Hermitian matrices,
and so $[(u^\epsilon)^2]_r - c|z|^2$ is a $\Gamma$-subsolution. But this
converges locally uniformly to $[u^2]_r - c|z|^2$, which is therefore
also a $\Gamma$-subsolution. 

Consider now the set
\[ U' = \{ z\,:\, 2u(z) < [u^2]_r(z) - c|z|^2 + [u]_\rho(z) -
4/3\}, \]
which satisfies $U' \subset U$, and in addition $U'$ is bounded, by
the assumption that $|u| \leq 1$. This contradicts the comparison
result Lemma~\ref{lem:comp}, since $u$ is a 
$\Gamma$-solution and as we have seen, $[u^2]_r - c|z|^2 + [u]_p$ is a
$\Gamma$-subsolution on the set $U$. 
\end{proof}

\section{Blowup argument}\label{sec:blowup}

We now complete the proof of Theorem~\ref{thm:main} using a blowup argument analogous
to that in \cite{DK12}, using the Liouville-type theorem,
Theorem~\ref{thm:DK}. 

\begin{proof}[Proof of Theorem~\ref{thm:main}]
  Suppose that as in the introduction, $(M,\alpha)$ is Hermitian,
  $\chi$ is a real $(1,1)$-form, and $g = \chi + \ddbar u$ satisfies $F(A)
  = h$, where $A^{ij} = \alpha^{j\bar p}g_{i\bar p}$. We use
  Proposition~\ref{prop:HMW}, together with a contradiction argument
  to obtain an estimate for $|\nabla u|$, depending on the $C^2$-norms
  of $\alpha, \chi, h$ and the subsolution $\underline{u}$,
  which in turn will imply an estimate for $\Delta u$. The
  $C^{2,\alpha}$-estimate follows from this by the Evans-Krylov
  theory. 

  To argue by contradiction, suppose that $F(A) = h$, but
  \[ \label{eq:gradbound}\sup_M |\nabla u|^2 = |\nabla u(p)|^2 = N, \]
  for some large $N$. Proposition~\ref{prop:HMW} implies that we have
  \[ \label{eq:Dest} |\d\overline{\d} u|_\alpha \leq CN \]
  for a fixed constant $C$. Let $\widetilde{\alpha} = N\alpha$.
  We can choose coordinates $z_1,\ldots, z_n$ centered at $p$, such
  that in these coordinates $\widetilde{\alpha}, \chi, h$ satisfy
  \[  \begin{aligned} \widetilde{\alpha}_{i\bar j} &= \delta_{i\bar j}
    + O(N^{-1}|z|), \\
    \chi_{i\bar j} &= O(N^{-1}), \\
    h &= h(p) + O(N^{-1}|z|),
    \end{aligned}\]
  and the $z_i$ are defined for $|z_i| < O(N^{1/2})$.  The inequality
  \eqref{eq:Dest} implies that $|\d\overline{\d} u|_{\widetilde{\alpha}} \leq C$,
  and since $\widetilde{\alpha}$ is approximately Euclidean on the
  ball of radius $O(N^{1/2})$, we obtain a uniform bound
  \[ \label{eq:c1alpha}\Vert u\Vert_{C^{1,\alpha}} < C', \]
  on this ball, in these coordinates. The equation $F(A) = h$ implies that
  \[ \label{eq:fNl} f\left(N\lambda\big[ \widetilde{\alpha}^{j\bar p} (\chi_{i\bar p} +
  u_{i\bar p})\big]\right) = h(z), \]
  where $f:\Gamma\to\mathbf{R}$ defines the operator $F$. Since
  we have a fixed bound on $u_{i\bar p}$, while $\chi_{i\bar p}$ is
  going to zero and $\widetilde{\alpha}^{j\bar p}$ is approaching the
  identity matrix, we obtain
  \[ \label{eq:latilde} \lambda\big[ \widetilde{\alpha}^{j\bar p} (\chi_{i\bar p} +
  u_{i\bar p})\big] = \lambda(u_{i\bar j}) + O(N^{-1}|z|).\]

  Suppose now that we have a sequence of such $\alpha, \chi, h$ and
  the subsolution $\underline{u}$ all bounded in $C^2$, and with $|u|$ uniformly
  bounded so that
  Proposition~\ref{prop:HMW} can be applied uniformly, and the
  constant $N$ in \eqref{eq:gradbound} gets larger and larger. The
  coordinates $z_i$ will then be defined on larger and larger balls,
  and using the estimate \eqref{eq:c1alpha} we can choose a
  subsequence converging uniformly in $C^{1,\alpha}$ to a limit
  $v:\mathbf{C}^n\to\mathbf{R}$. By the construction we will have
  global bounds $|v|, |\nabla v| < C$, and $|\nabla v(0)| = 1$.

  The proof will be completed by showing that $v$ is a $\Gamma$-solution, in the sense of
  Definition~\ref{defn:Gammasol}, since that will contradict
  Theorem~\ref{thm:DK}.
  To see this, suppose first that we have a
  $C^2$-function $\psi$, such that $\psi \geq v$, and $\psi(z_0) =
  v(z_0)$ for some point $z_0$. We need to show that
  $\lambda(\psi_{i\bar j}(z_0))\in \overline{\Gamma}$. By the construction
  of $v$, for any $\epsilon > 0$ we can find a $u$ as above,
  corresponding to a sufficiently large $N$, a number $a$ with $|a| < \epsilon$,
  and a point $z_1$ with $|z_1 - z_0| < \epsilon$, such that 
  \[ \psi + \epsilon |z - z_0|^2 + a \geq u \text{ on $B_1(z_0)$, with equality at }
  z_1. \]
  This implies that $\psi_{i\bar j}(z_1) + \epsilon \delta_{i\bar j}
  \geq u_{i\bar j}(z_1)$. 
  From \eqref{eq:latilde}, and the fact that $\Gamma_n\subset \Gamma$
  we obtain that for large $N$, $\lambda\big[\psi_{i\bar j}(z_1)\big]$ will be within
  $2\epsilon$ of $\Gamma$. Letting $\epsilon \to 0$ we find that
  $\lambda\big[\psi_{i\bar j}(z_0)\big]\in \overline{\Gamma}$. 

Suppose now that we have a $C^2$-function $\psi$ such that $\psi \leq
v$ and $\psi(z_0) = v(z_0)$. We need to show $\lambda\big[\psi_{i\bar
  j}(z_0)\big]\in \mathbf{R}^n\setminus\Gamma$. As above, for any
$\epsilon > 0$ we can find a $u$ corresponding to large $N$, and $a,
z_1$, such that
\[ \psi - \epsilon |z-z_0|^2 + a \leq u \text{ on $B_1(z_0)$, with
  equality at } z_1. \]
This implies that $\psi_{i\bar j}(z_1) - \epsilon \delta_{i\bar j}
\leq u_{i\bar j}(z_1)$. This implies that if  $\lambda(\psi_{i\bar j}(z_1) -
3\epsilon\delta_{i\bar j}) \in \Gamma$, then we will have
$\lambda(u_{i\bar j}) \in \Gamma + 2\epsilon\mathbf{1}$. Using
\eqref{eq:latilde}, if $N$ is sufficiently large, we will have
\[ \lambda\big[\widetilde{\alpha}^{j\bar p}(\chi_{i\bar p} + u_{i\bar
  p})\big] \in \Gamma + \epsilon\mathbf{1}. \]
Finally, by Lemma~\ref{lem:struct} part (a), our assumptions for $f$
in the introduction imply  that
if $N$ is sufficiently large, then we cannot have \eqref{eq:fNl},
since we have a fixed bound for $h$, which must be less than
$\sup_\Gamma f$. It follows that we cannot have $\lambda(\psi_{i\bar
  j}(z_1)) \in \Gamma + 3\epsilon\mathbf{1}$. Letting $\epsilon\to 0$
we will have $z_1\to z_0$, and so $\lambda(\psi_{i\bar j}(z_0)) \in
\mathbf{R}^n\setminus\Gamma$. This completes the proof that $v$ is a
$\Gamma$-solution.   
\end{proof}

\section{Hessian quotient equations}\label{sec:ex}
In this section we prove Corollary~\ref{cor:hessquot} as an application of
Theorem~\ref{thm:main}, and we also discuss equations for
$(n-1)$-plurisubharmonic functions. As we mentioned in the
introduction it is somewhat difficult to formulate very general
existence results, in contrast to the case of the Dirichlet problem in
\cite{CNS3}, because on a compact manifold the constant functions are
not in the image of the linearized operator of
Equation~\eqref{eq:eqn}. In particular if we consider only equations
with constant right hand side, $F(A) = c$, then a solution can only
exist for a unique constant. If we do not know a priori what the right
constant is, then we cannot ensure that along a suitable continuity
path we have a $\mathcal{C}$-subsolution for the whole path. 
This issue does not arise when any admissible function is a
$\mathcal{C}$-subsolution, which is the case for
the complex Monge-Amp\`ere and Hessian equations for instance. We
therefore have the following. 
\begin{prop}\label{prop:Hessian}
  Let $(M,\alpha)$ be compact, Hermitian, let $\chi$ be a
  $k$-positive real $(1,1)$-form on $M$ and let $1 \leq k \leq n$.
  Given any smooth function $H$ on $M$, we can find a constant $c$ and
  a function $u$, such that the form $\omega =
  \chi + \ddbar u$ satisfies the equation
  \[ \label{eq:hessianeqn}\omega^k \wedge \alpha^{n-k} = e^{H+c} \alpha^n. \]
\end{prop}
Note that for $k=1$ this is the Poisson equation whose solution is
standard, while for $k=n$ it
is the complex Monge-Amp\`ere equation, which was solved on K\"ahler
manifolds by Yau~\cite{Yau78} and by Tosatti-Weinkove~\cite{TW10} on
Hermitian manifolds. For
$1 < k < n$ it was solved by Dinew-Kolodziej~\cite{DK12} on K\"ahler
manifolds, and by Sun~\cite{Sun14_2} on Hermitian manifolds (see also
Zhang~\cite{DZ15}).  For the
reader's convenience we present the proof here. 
\begin{proof}
  We can write the equation in the form $F(A) = h$, for a positive
  function $H$ depending on $h$, where $F$ is defined by the function
  \[ f = \log \sigma_k \]
  on the $k$-positive cone $\Gamma_k$:
\[ \Gamma_k = \{\lambda\,:\, \sigma_1(\lambda),\ldots,
\sigma_k(\lambda) > 0\}. \]
  This satisfies the structural conditions that we use (see
  Spruck~\cite{Spruck05}). In addition $\underline{u}=0$ is a
  subsolution if $\chi$ is any $k$-positive form. We can see this
  using Remark~\ref{rem:subsol2} together with the fact that for any
  $\mu=(\mu_1,\ldots, \mu_n)\in \Gamma_k$ we have
  \[ \lim_{t\to\infty} \sigma_k(\mu_1,\ldots, \mu_{n-1},  t) = \infty. \]

  We therefore have great flexibility in setting up a continuity
  method. For instance we can let $H_0$ be the function such that
  \[ \chi^k \wedge \alpha^{n-k} = e^{H_0} \alpha^n, \]
  and then solve the family of equations
  \[ \label{eq:cont2}
  \log \frac{(\chi + \ddbar u_t)^k\wedge \alpha^{n-k}}{\alpha^n} =
  tH + (1-t)H_0 + c_t, \]
  for $t\in [0,1]$, where $c_t$ are constants. 
  For $t=0$ a solution is $u_0=0, c_0=0$. Openness follows from the
  fact that if $L$ denotes the linearized operator at any $t\in [0,1]$, then the operator
  \[  \begin{aligned}   C^{k,\beta} \times \mathbf{R} &\to
    C^{k-2,\beta} \\
    (v, c) &\mapsto Lv + c
  \end{aligned}\]
  is surjective. To obtain a priori estimates for the solutions we can
  first obtain bounds for $c_t$ from above and below by looking at the
  points where $u_t$ achieves its maximum and minimum in
  Equation~\eqref{eq:cont2}. Given this, Theorem~\ref{thm:main} gives
  higher order estimates. 
\end{proof}

We next focus on the Hessian quotient equation
\[ \label{eq:hqeq}\omega^l\wedge \alpha^{n-l} = c\omega^k \wedge \alpha^{n-k}, \]
where $(M,\alpha)$ is K\"ahler, $1\leq l < k \leq n$, and $\omega =
\chi + \ddbar u$ with a fixed, closed background form $\chi$. We assume that
the constant $c$ is chosen so that
\[ \label{eq:cdef} c = \frac{\int_M \chi^l\wedge \alpha^{n-l}}{\int_M \chi^k\wedge
  \alpha^{n-k}}. \]
The standard way of writing our equation would be to use the function
$g = (\sigma_k / \sigma_l)^{1/(k-l)}$ on $\Gamma_k$. 
This function satisfies the required conditions (see
Spruck~\cite{Spruck05}),  however it seems not to be well adapted to setting up a continuity
method. Instead we will write the equation in the form
\[ -\frac{\omega^l \wedge \alpha^{n-l}}{\omega^k \wedge \alpha^{n-k}}
= -c, \]
which is the same as $F(A) = -c$ with $F$ defined by the function
\[ f = -\frac{\binom{n}{l}^{-1} \sigma_l}{\binom{n}{k}^{-1} \sigma_k}. \]
Note that again $f$ is concave, since $f = -g^{-(k-l)}$. We will use a
continuity method interpolating between this, and the Hessian
equation, given by the function 
\[ f_0 = -\frac{1}{\binom{n}{k}^{-1} \sigma_k}. \]
In other words, we will try to solve the equation
\[ \label{eq:conteq}
  t \frac{\omega^l\wedge \alpha^{n-l}}{\omega^k \wedge \alpha^{n-k}} +
  (1-t) \frac{\alpha^n}{\omega^k \wedge \alpha^{n-k}} = c_t,
  \]
  for $t\in [0,1]$. 

Corollary~\ref{cor:hessquot} follows from the following.
\begin{prop}
  Suppose that $\chi$ is a closed $k$-positive form, satisfying
  \[ \label{eq:aa2} kc\chi^{k-1}\wedge \alpha^{n-k} - l\chi^{l-1}\wedge \alpha^{n-l}
  > 0, \]
  in the sense of positivity of $(n-1,n-1)$-forms, where $c$ is
  defined by \eqref{eq:cdef}. The Equation~\eqref{eq:conteq} has a
  solution for all $t\in [0,1]$, for suitable $c_t$, such that
  $c_1=c$. 
\end{prop}
\begin{proof}
  For $t=0$ we can solve the equation using
  Proposition~\ref{prop:Hessian}, and openness follows in the same way
  as in the proof of that proposition. It remains to obtain a priori
  estimates.

  Note first of all, that by integrating \eqref{eq:conteq} on $M$ with
  respect to $\omega^k\wedge\alpha^{n-k}$, we find that $c_t \geq tc$
  for $t\in [0,1]$. Writing Equation~\eqref{eq:conteq} in the form
  \[ f_t(\lambda) = -t\frac{\binom{n}{l}^{-1}\sigma_l}{\binom{n}{k}^{-1}\sigma_k} - (1-t)
  \frac{1}{\binom{n}{k}^{-1}\sigma_k} = -c_t, \]
  the equation satisfies our structural assumptions, and we claim that
  $\underline{u}=0$ is a $\mathcal{C}$-subsolution for it. For this,
  let $\mu_i$ denote the eigenvalues of $\alpha^{j\bar p}\chi_{j\bar
    p}$. By Remark~\ref{rem:subsol2} we just need to check that if
  $\mu'$ denotes any $(n-1)$-tuple from the $\mu_i$, then
  \[ \lim_{R\to\infty} f_t(\mu', R) > -c_t, \]
  which by the formula for $f_t$ means
  \[ \label{eq:aa1} -t
  \frac{\binom{n}{l}^{-1}\sigma_{l-1}(\mu')}{\binom{n}{k}^{-1}\sigma_{k-1}(\mu')}
  > -c_t. \]
  To rewrite this in terms of the forms $\chi, \alpha$, note that at
  any given point, if we restrict ourselves to the subspace of the
  tangent space of $M$ spanned by the eigenvectors corresponding to
  $\mu'$, then on this subspace
  \[ \sigma_{i-1}(\mu') = \binom{n-1}{i-1}\frac{\chi^{i-1}\wedge
  \alpha^{n-i}}{\alpha^{n-1}} \]
  for all $i$. Applying this to $i=l, k$ we find that \eqref{eq:aa1}
  is equivalent to the inequality
  \[ kc_t \chi^{k-1}\wedge \alpha^{n-k} - l t \chi^{l-1}\wedge
  \alpha^{n-l} > 0\]
  for $(n-1,n-1)$-forms. Since $\chi$ is $k$-positive and $c_t \geq
  tc$, this follows from \eqref{eq:aa2}. Theorem~\ref{thm:main} will
  then give uniform estimates for $t$ in any compact interval $[c, 1]$
  for $ c> 0$.  
\end{proof}

It is an interesting problem to
find geometric conditions which ensure the existence of a
$\mathcal{C}$-subsolution. In analogy with the conjecture in
\cite{LSz13} regarding the case when $k=n, l=n-1$, it is natural to
conjecture the following.

\begin{conj}\label{conj:hessq}
  Suppose that $\chi$ is a closed $k$-positive form.
  Then we can find a $k$-positive $\chi'\in [\chi]$ satisfying the
  inequality \eqref{eq:aa2} with $\chi'$ instead of $\chi$ if and only
  if for all subvarieties $V\subset M$ of dimension $p=n-l, \ldots, n-1$
  we have
  \[ \int_V c\frac{k!}{(k-n+p)!} \chi^{k-n+p}\wedge \alpha^{n-k} -
  \frac{l!}{(l-n+p)!} \chi^{l-n+p} \wedge \alpha^{n-l} > 0. \]
\end{conj}

As we mentioned in the introduction, this conjecture has recently been
resolved in \cite{CSz14} in the case when $M$ is a toric manifold, and
$k=n, l=n-1$, but cases beyond this are mostly open. Another
related problem is to characterize real (1,1)-classes which admit
$k$-positive representatives, in analogy with the result of
Demailly-Paun~\cite{DP04} in the case $k=n$. Such $k$-positive
representatives provide $\mathcal{C}$-subsolutions for the
Hessian equation \eqref{eq:hessianeqn}. 

An equation related to the complex Monge-Amp\`ere equation was
introduced by Fu-Wang-Wu~\cite{FWW10}. Given Hermitian metrics $\alpha,
\eta$ and a function $h$, the equation can be written as 
\[ \label{eq:formCY} \det\left( \eta_{i\bar j} + \frac{1}{n-1}\big[ (\Delta u)\alpha_{i\bar
    j} - u_{i\bar j}\big]\right) = e^{h} \det \alpha, \]
where we require that the form inside the determinant on the left hand
side is positive definite, and $\Delta$ denotes the (complex) Laplacian with
respect to $\alpha$. Let us define the form
\[ \label{eq:chidefn}
\chi_{i\bar j} = (\mathrm{tr}_\alpha\eta) \alpha_{i\bar j} - (n-1)\eta_{i\bar
  j}, \]
and denote by $T$ the map $T(A) = \frac{1}{n-1}(\mathrm{Tr}(A)I - A)$
on matrices, where $I$ is the identity. Write $g_{i\bar j} =
\chi_{i\bar j} + u_{i\bar j}$, and $A^i_j = \alpha^{i\bar p}g_{j\bar
  p}$ as before. Then Equation~\eqref{eq:formCY} is equivalent to the
equation
\[ \log\det(T(A)) = h. \]
In other words, we can write our equation as $F(A) = h$, where $F$ is
determined by the symmetric function
\[ f(\lambda) = \log \prod_{k=1}^n T(\lambda)_k, \]
where  abusing
notation we denote by $T:\mathbf{R}^n\to\mathbf{R}^n$ the map with
components
\[ T(\lambda)_k = \frac{1}{n-1}\sum_{i\ne k} \lambda_i, \]
giving the map on eigenvalues corresponding to the matrix map $T$
above. 
The conditions $(i), (ii), (iii)$ in the
introduction hold for this function $f$ on the pre-image
$T^{-1}(\Gamma_n)$ of the positive orthant under $T$, and so using
Theorem~\ref{thm:main} we can obtain a priori estimates. Moreover,
from Remark~\ref{rem:subsol2}, we see that $\underline{u}=0$ is a
$\mathcal{C}$-subsolution, whenever 
$\eta$ is positive definite in Equation~\eqref{eq:formCY}, since then
$f(\lambda)\to\infty$ if we let any one component of $\lambda$ go to
infinity.  In
particular the same argument as in the proof of
Proposition~\ref{prop:Hessian} can be used to show that if $\eta$ is
positive definite, then Equation~\eqref{eq:formCY} admits a smooth
solution $u$ for any smooth $h$, up to adding a constant to $h$. 
This recovers the result of
Tosatti-Weinkove~\cite{TW13_1} in the case when $\alpha$ is K\"ahler,
and Tosatti-Weinkove~\cite{TW13} when $\alpha$ is Hermitian (see also
Fu-Wang-Wu~\cite{FWW10_1} for an earlier result under a non-negative
curvature assumption for $\alpha$). 

One can also consider more general equations in this vein, for
instance those given by the functions
\[ f(\lambda) = \log
\frac{\sigma_k(T(\lambda))}{\sigma_l(T(\lambda))} \]
for $0\leq l < k \leq n$,
where $T:\mathbf{R}^n\to\mathbf{R}^n$ is the map above. 
Using that $T$ maps $\Gamma_n$ into itself, one sees that these $f$
satisfy the conditions $(i),(ii),(iii)$ in the introduction on the preimage
$T^{-1}(\Gamma_k)$ of the $k$-positive cone
$\Gamma_k$. When $l > 0$, the subsolution
condition is non-trivial (i.e. it depends on $h$), 
so we do not expect to be able to solve
$F(A) = h$ for all $h$, even up to adding a constant to $h$. 
However when $l=0$, and the background form
 $\chi$ is such  that $\underline{u}=0$ is a
$\mathcal{C}$-subsolution, then $0$ will be a
$\mathcal{C}$-subsolution for all $h$, since in this case we have
$\lim_{t\to\infty}f(\lambda + t\mathbf{e}_i)=\infty$ for all $i$. 
In particular, 
 the same argument as the proof of
Proposition~\ref{prop:Hessian} implies the following, which answers a
question raised in Tosatti-Wang-Weinkove-Yang~\cite{TWWY14}. 

\begin{prop}
  Let $(M,\alpha)$ be a compact Hermitian manifold, $\eta$ a
  $k$-positive real $(1,1)$-form, and $1\leq k\leq n$. 
  For any smooth function $H$ on $M$
  we can find a constant $c$ and a function $u$ satisfying the
  equation
  \[ \left(\eta + \frac{1}{n-1}\big[ (\Delta u)\alpha - \ddbar
    u\big]\right)^k \wedge \alpha^{n-k} = e^{H+c}\alpha^n. \]
Here $\Delta$ denotes the complex Laplacian with respect to $\alpha$,
i.e. $\Delta u = \alpha^{p\bar q}u_{p\bar q}$. 
\end{prop}

To see this, note that as for Equation~\eqref{eq:formCY}, we write the
equation in the form $F(A)=h$ where $A^i_j = \alpha^{i\bar
  p}(\chi_{j\bar p} + u_{j\bar p})$ in terms of the background form
$\chi$ given by \eqref{eq:chidefn}. If $\eta$ is $k$-positive, then
the eigenvalues of $\alpha^{i\bar p}\chi_{j\bar p} =
T^{-1}(\alpha^{i\bar p}\eta_{j\bar p})$ lie in the preimage
$T^{-1}(\Gamma_k)$ of the $k$-positive cone, and so $0$ is a
$\mathcal{C}$-subsolution.

\section{Equations on Riemannian manifolds}\label{sec:Riemannian} 
In this section we will describe how
the methods in this paper apply to equations analogous to
\eqref{eq:eqn} on Riemannian manifolds as well.  So in this section we
let $(M, \alpha)$ be a compact Riemannian manifold and $\chi$ a fixed
tensor of type (0,2). Suppose we are interested in solving the
equation
\[ F(A) = h \]
where analogously to before, $A$ is the endomorphism of the tangent
bundle given by $A^i_j = \alpha^{ip}(\chi_{jp} + u_{jp})$ for the
unknown function $u$, and $u_{jp}$ denote covariant derivatives with
respect to $\alpha$. This endomorphism is symmetric with respect to
the inner product defined by $\alpha$ at each point, and as before
$F(A) = f(\lambda(A))$ in terms of the eigenvalues $\lambda(A)$ of
$A$. We assume that $f$ satisfies the structural conditions (i), (ii),
(iii) from the introduction.

Everything that we have done in the Hermitian case applies in
this Riemannian setting as well, with almost exactly the same
proof. There are no torsion terms, but instead there are some extra curvature
terms obtained when commuting derivatives. We briefly note some of the
differences. In \eqref{eq:ff2} we obtain
\[ u_{11kk} = u_{kk11} + O(\lambda_1), \]
while in \eqref{eq:l7} we obtain
\[ g_{11k} = g_{k11} + O(K^{1/2}), \]
and so
\[ |g_{11k}|^2 \leq |g_{k11}|^2 + C_0(K^{1/2}|g_{1\bar 1k}| + K). \]
The upshot is that we have the following analogous inequality to
\eqref{eq:l5}:
\[ \label{eq:ff3}
L(\log \widetilde{\lambda}_1) \geq \frac{ -
  F^{pq,rs}g_{pq1}g_{rs1}}{\lambda_1} -
\frac{F^{kk}|g_{k11}|^2}{\lambda_1^2} - C_0(\mathcal{F} +
K^{1/2}|g_{1\bar 1k}|). \]
Using this the rest of the argument in the $C^2$-estimate is essentially identical. 

The Riemannian case has the distinct advantage in that when using the
concavity of the operator as in \eqref{eq:concaveF}, we get twice as
many useful terms as in the complex case:
\[ F^{ij,rs}u_{ij1}u_{rs1} \leq  2 \sum_{k > 1}
\frac{f_1-f_k}{\lambda_1 - \lambda_k} |u_{k11}|^2. \]
In the complex case the corresponding extra terms are $|u_{1\bar
  k1}|^2$, which do not appear to be useful in the estimates. A
consequence of this is that in the real case it is more straight
forward to control the bad negative term involving $|g_{11k}|^2$. 
See for instance
Guan-Jiao~\cite{GJ14} for such estimates. 

In the proof of Theorem~\ref{thm:DK}, when we start the induction
argument, the case $n=1$ corresponds to bounded linear functions on
$\mathbf{R}$ being constant, while the case $\Gamma= \Gamma_n$
corresponds to bounded convex functions on $\mathbf{R}^n$ being
constant.

Just as before, a function $\underline{u}$ is a
$\mathcal{C}$-subsolution for the equation $F(A)=h$, if the matrix
$B^i_j = \alpha^{ip}(\chi_{jp} + \underline{u}_{jp})$ is such that the
set $(\lambda(B) + \Gamma_n)\cap \d\Gamma^{h(x)}$ is bounded at each
$x\in M$. We then have the following.

\begin{prop} \label{prop:r2}
 Suppose that there exists a $\mathcal{C}$-subsolution
  $\underline{u}$ for the equation $F(A)=h$ as above. Normalizing $u$
  so that $\sup_M u=0$, we have a priori estimates $\Vert
  u\Vert_{C^{2,\alpha}} < C$, with constant depending on the
  background data as well as the subsolution $\underline{u}$. 
\end{prop}

This result generalizes several earlier results on these types of
equations on compact Riemannian manifolds, such as Li~\cite{Li90},
Delano\"e~\cite{Del03} who made non-negative curvature assumptions,
and Urbas~\cite{Urb04}, Guan~\cite{Guan14}, who have stronger
structural assumptions. In particular in Urbas~\cite{Urb04} the
question of solving the Hessian quotient equations analogous to
\eqref{eq:hqeq} on compact
Riemannian manifolds is raised. This is formulated as the equation
\[ \log F(A) = h + c, \]
where $h$ is a given function, the function $u$ and constant $c$
are the unknowns, and $F$ is given by the function
\[ f = \left(\frac{\sigma_k}{\sigma_l}\right)^{\frac{1}{k-l}}, \]
for some  $1 \leq l < k \leq n$. 

In analogy with the K\"ahler case, it is natural to
expect that these equations do not always have a solution, but it
seems to be difficult to formulate a condition as precise as that in
Conjecture~\ref{conj:hessq}. Instead we formulate a general existence
result, focusing for simplicity on equations of the form $F(A) = c$
with constant $c$.  Note
that if $F$ is homogeneous and positive, then the restriction to constant right
hand side can be removed by scaling the metric $\alpha$. 

\begin{prop} \label{prop:riemexist}
  Suppose that $\sup_{\d\Gamma} f = -\infty$, and
  $sup_\Gamma f = \infty$. Let $h_0 =   F(\alpha^{ip}\chi_{jp})$. If the equation $F(A)
  = \sup_M h_0$ admits a $\mathcal{C}$-subsolution $\underline{u}$,
  then the equation $ F(A) =
  c$ has a solution for some constant $c$. 
\end{prop}
\begin{proof}We want to use the
  continuity method to solve the equations
  \[\label{eq:cc2} F(A) = c_t + (1-t)h_0, \]
  for $t\in [0,1]$ with constants $c_t$. For $t=0$ the solution is
  $u=c_0 = 0$, and openness follows using the implicit function
  theorem as before.

  To find a priori estimates, the only thing we need is
  $\mathcal{C}$-subsolutions for each $t$, and we need to make sure
  that the range of the right hand side $c_t + (1-t)h_0$ is contained
  in a compact subset of the range of $f$ in order to obtain uniform
  constants.  Suppose that $u$ is a solution of \eqref{eq:cc2} and $u$
  achieves its minimum and maximum at $p\in M$ and $q\in M$
  respectively. We then have $F(A) \geq F(\alpha^{ip}\chi_{jp})$ at
  $p$ and $F(A)\leq F(\alpha^{ip}\chi{jp})$ at $q$. It follows that
  \[ h_0(p) \leq c_t + (1-t) h_0(p), \]
i.e. $c_t \geq  th_0(p)$, and similarly $c_t \leq th_0(q)$. In
particular we obtain upper and lower bounds for $c_t + (1-t)h_0$,
whose range is then in a compact subset of the range of $f$ by our
assumption for $f$. More precisely at any $x\in M$ we have
\[ c_t + (1-t)h_0(x) \leq th_0(q) + (1-t)h_0(x) \leq \sup_M h_0, \]
which implies that  $\underline{u}$ is a $\mathcal{C}$-subsolution for
Equation~\eqref{eq:cc2} for each $t$. Proposition~\ref{prop:r2} then
implies the required estimates. 
\end{proof}

\subsection*{Acknowledgements}
I would like to thank Tristan Collins, Bo
Guan, Duong Phong, Ovidiu Savin, Wei Sun, and
 Valentino Tosatti  for helpful comments and
discussions. The author is supported by NSF
grants DMS-1306298 and DMS-1350696.

\bibliographystyle{siam}
\bibliography{../mybib}

\begin{thebibliography}{10}

\bibitem{Bl05_1}
{\sc Z.~B\l{}ocki}, {\em On uniform estimate in {C}alabi-{Y}au theorem}, Sci.
  China Ser. A, 48 (2005), pp.~244--247.

\bibitem{CC95}
{\sc L.~Caffarelli and X.~Cabr\'e}, {\em Fully nonlinear elliptic equations},
  vol.~43 of American Mathematical Society Colloquium Publications, American
  Mathematical Society, Providence, RI, 1995.

\bibitem{CNS3}
{\sc L.~Caffarelli, L.~Nirenberg, and J.~Spruck}, {\em The {D}irichlet problem
  for nonlinear second-order elliptic equations {III}: {F}unctions of the
  eigenvalues of the {H}essian}, Acta Math., 155 (1985), pp.~261--301.

\bibitem{Ch87}
{\sc P.~Cherrier}, {\em \'equations de {M}onge-{A}mp\`ere sur les vari\'et\'es
  {H}ermitiennes compactes}, Bull. Sci. Math. (2), 111 (1987), pp.~343--385.

\bibitem{CW01}
{\sc K.-S. Chou and X.-J. Wang}, {\em A variational theory of the {H}essian
  equation}, Comm. Pure Appl. Math., 54 (2001), pp.~1029--1064.

\bibitem{CSz14}
{\sc T.~Collins and G.~Sz\'ekelyhidi}, {\em Convergence of the {J}-flow on
  toric manifolds}, arXiv:1412.4809.

\bibitem{Del03}
{\sc P.~Delano\"e}, {\em Hessian equations on compact non-negatively curved
  {R}iemannian manifolds}, Calc. Var. Partial Differential Equations, 16
  (2003), pp.~165--176.

\bibitem{DP04}
{\sc J.-P. Demailly and M.~Paun}, {\em Numerical characterization of the
  {K}\"ahler cone of a compact {K}\"ahler manifold}, Ann. of Math. (2), 159
  (2004), pp.~1247--1274.

\bibitem{DK12}
{\sc S.~Dinew and S.~Ko\l{}odziej}, {\em Liouville and {C}alabi-{Y}au type
  theorems for complex {H}essian equations}, arXiv:1203.3995.

\bibitem{Ev82}
{\sc L.~C. Evans}, {\em Classical solutions of fully nonlinear, convex, second
  order elliptic equations}, Comm. Pure Appl. Math., 25 (1982), pp.~333--363.

\bibitem{FL12_1}
{\sc H.~Fang and M.~Lai}, {\em On the geometric flows solving {K}{\"a}hlerian
  inverse {$\sigma_k$} equations}, Pacific J. Math., 258 (2012), pp.~291--304.

\bibitem{FLM11}
{\sc H.~Fang, M.~Lai, and X.~Ma}, {\em On a class of fully nonlinear flows in
  {K}\"ahler geometry}, J. Reine Angew. Math., 653 (2011), pp.~189--220.

\bibitem{FWW10_1}
{\sc J.~Fu, Z.~Wang, and D.~Wu}, {\em Form-type {C}alabi-{Y}au equations on
  {K}\"ahler manifolds of nonnegative orthogonal bisectional curvature},
  arXiv:1010.2022.

\bibitem{FWW10}
\leavevmode\vrule height 2pt depth -1.6pt width 23pt, {\em Form-type
  {C}alabi-{Y}au equations}, Math. Res. Lett., 17 (2010), pp.~887--903.

\bibitem{Ge96}
{\sc C.~Gerhardt}, {\em Closed {W}eingarten hypersurfaces in {R}iemannian
  manifolds}, J. Differential Geom., 43 (1996), pp.~612--641.

\bibitem{GT98}
{\sc D.~Gilbarg and N.~S. Trudinger}, {\em Elliptic partial differential
  equations of second order}, Classics in Mathematics, Springer-Verlag, Berlin,
  reprint of the 1998~ed., 2001.

\bibitem{Guan14_1}
{\sc B.~Guan}, {\em The {D}irichlet problem for fully nonlinear elliptic
  equations on {R}iemannian manifolds}, arXiv:1403.2133.

\bibitem{Guan14}
\leavevmode\vrule height 2pt depth -1.6pt width 23pt, {\em Second-order
  estimates and regularity for fully nonlinear elliptic equations on
  {R}iemannian manifolds}, Duke Math. J., 163 (2014), pp.~1491--1524.

\bibitem{GJ14}
{\sc B.~Guan and H.~Jiao}, {\em Second order estimates for hessian type fully
  nonlinear elliptic equations on riemannian manifolds}, arXiv:1401.7391.

\bibitem{GL09}
{\sc B.~Guan and Q.~Li}, {\em Complex {M}onge-{A}mp\`ere equations on
  {H}ermitian manifolds}, arXiv:0906.3548.

\bibitem{GS13}
{\sc B.~Guan and W.~Sun}, {\em On a class of fully nonlinear elliptic equations
  on {H}ermitian manifolds}, arXiv:1301.5863.

\bibitem{Ha96}
{\sc A.~Hanani}, {\em {\'E}quations du type de {M}onge-{A}mp\`ere sur les
  vari\'et\'es hermitiennes compactes}, J. Funct. Anal., 137 (1996),
  pp.~49--75.

\bibitem{HMW10}
{\sc Z.~Hou, X.-N. Ma, and D.~Wu}, {\em A second order estimate for complex
  {H}essian equations on a compact {K}\"ahler manifold}, Math. Res. Lett., 17
  (2010), pp.~547--561.

\bibitem{Kry82}
{\sc N.~V. Krylov}, {\em Boundedly nonhomogeneous elliptic and parabolic
  equations}, Izvestia Akad. Nauk. SSSR, 46 (1982), pp.~487--523.

\bibitem{LSz13}
{\sc M.~Lejmi and G.~Sz\'ekelyhidi}, {\em The {J}-flow and stability},
  arXiv:1309.2821.

\bibitem{LS04_1}
{\sc S.-Y. Li}, {\em On the {D}irichlet problems for symmetric function
  equations of the eigenvalues of the complex {H}essian}, Asian J. Math., 8
  (2004), pp.~87--106.

\bibitem{Li90}
{\sc Y.-Y. Li}, {\em Some existence results for fully nonlinear elliptic
  equations of {M}onge-{A}mp\`ere type}, Comm. Pure Appl. Math., 43 (1990),
  pp.~233--271.

\bibitem{PS09}
{\sc D.~Phong and J.~Sturm}, {\em The {D}irichlet problem for degenerate
  complex {M}onge-{A}mp\`ere equations}, arXiv:0904.1898.

\bibitem{PSS12_1}
{\sc D.~H. Phong, J.~Song, and J.~Sturm}, {\em Complex {M}onge-{A}mp\`ere
  equations}, in Surveys in differential geometry. {V}ol. {XVII}, vol.~17, Int.
  Press, Boston, MA, 2012, pp.~327--410.

\bibitem{SW04}
{\sc J.~Song and B.~Weinkove}, {\em On the convergence and singularities of the
  {J}-flow with applications to the {M}abuchi energy}, Comm. Pure Appl. Math.,
  61 (2008), pp.~210--229.

\bibitem{Spruck05}
{\sc J.~Spruck}, {\em Geometric aspects of the theory of fully nonlinear
  elliptic equations}, in Global theory of minimal surfaces, vol.~2, Amer.
  Math. Soc., Providence, RI, 2005, pp.~283--309.

\bibitem{Sun14_2}
{\sc W.~Sun}, {\em Complex {H}essian equations on closed {H}ermitian
  manifolds}, preprint.

\bibitem{Sun13}
\leavevmode\vrule height 2pt depth -1.6pt width 23pt, {\em On a class of fully
  nonlinear elliptic equations on closed {H}ermitian manifolds},
  arXiv:1310.0362.

\bibitem{Sun14}
\leavevmode\vrule height 2pt depth -1.6pt width 23pt, {\em On a class of fully
  nonlinear elliptic equations on closed {H}ermitian manifolds {II}:
  {$L^\infty$} estimate}, arXiv:1407.7630.

\bibitem{Sun14_1}
\leavevmode\vrule height 2pt depth -1.6pt width 23pt, {\em On uniform estimate
  of complex elliptic equations on closed {H}ermitian manifolds},
  arXiv:1412.5001.

\bibitem{TWWY14}
{\sc V.~Tosatti, Y.~Wang, B.~Weinkove, and X.~Yang}, {\em {$C^{2,\alpha}$}
  estimates for nonlinear elliptic equations in complex and almost complex
  geometry}, arXiv:1402.0554.

\bibitem{TW13}
{\sc V.~Tosatti and B.~Weinkove}, {\em Hermitian metrics, $(n-1,n-1)$ forms and
  {M}onge-amp\`ere equations}, arXiv:1310.6326.

\bibitem{TW13_1}
\leavevmode\vrule height 2pt depth -1.6pt width 23pt, {\em The
  {M}onge-{A}mp\`ere equation for (n-1)-plurisubharmonic functions on a compact
  {K}\"ahler manifold}, arXiv:1305.7511.

\bibitem{TW10}
\leavevmode\vrule height 2pt depth -1.6pt width 23pt, {\em The complex
  {M}onge-{A}mp\`ere equation on compact hermitian manifolds}, J. Amer. Math.
  Soc., 23 (2010), pp.~1187--1195.

\bibitem{TW10_1}
\leavevmode\vrule height 2pt depth -1.6pt width 23pt, {\em Estimates for the
  complex {M}onge-{A}mp\`ere equation on {H}ermitian and balanced manifolds},
  Asian J. Math., 14 (2010), pp.~19--40.

\bibitem{Tru95}
{\sc N.~Trudinger}, {\em On the {D}irichlet problem for {H}essian equations},
  Acta Math., 175 (1995), pp.~151--164.

\bibitem{Urb04}
{\sc J.~Urbas}, {\em Hessian equations on compact {R}iemannian manifolds}, in
  Nonlinear problems in mathematical physics and related topics, {II}, vol.~2
  of Int. Math. Ser. (N. Y.), Kluwer/Plenum, New York, 2002, pp.~367--377.

\bibitem{W06}
{\sc B.~Weinkove}, {\em On the {J}-flow in higher dimensions and the lower
  boundedness of the {M}abuchi energy}, J. Differential Geom., 73 (2006),
  pp.~351--358.

\bibitem{Yau78}
{\sc S.-T. Yau}, {\em On the {R}icci curvature of a compact {K}\"ahler manifold
  and the complex {M}onge-{A}mp\`ere equation {I}.}, Comm. Pure Appl. Math., 31
  (1978), pp.~339--411.

\bibitem{DZ15}
{\sc D.~Zhang}, {\em Hessian equations on closed {H}ermitian manifolds},
  arXiv:1501.03553.

\end{thebibliography}

\end{document}